\documentclass[a4paper]{amsart}
\usepackage[active]{srcltx}
\usepackage[all]{xy}
\usepackage{subfig}
\usepackage{hyperref}
\usepackage{mathrsfs}
\usepackage{esint}
\usepackage{amsmath}
\usepackage{amssymb,latexsym}
\usepackage{mathrsfs}
\usepackage{graphics}
\usepackage{latexsym}
\usepackage{psfrag}
\usepackage{import}
\usepackage{verbatim}
\usepackage{graphicx}
\usepackage[usenames]{color}
\usepackage{pifont,marvosym}
\usepackage[normalem]{ulem}
\usepackage{mathtools}

\theoremstyle{plain}
\newtheorem{lemma}{Lemma}[section]
\newtheorem{theorem}[lemma]{Theorem}
\newtheorem{proposition}[lemma]{Proposition}
\newtheorem{corollary}[lemma]{Corollary}
\newtheorem{conjecture}[lemma]{Conjecture}
\newtheorem{question}[lemma]{Question}

\theoremstyle{definition}

\newtheorem{definition}[lemma]{Definition}
\newtheorem{remark}[lemma]{Remark}

\numberwithin{equation}{section}

\newcommand{\R}{\mathbb{R}}
\newcommand{\RP}{\mathbb{RP}}
\newcommand{\bS}{\mathbb{S}}
\newcommand{\N}{\mathbb{N}}

\newcommand{\Z}{\mathbb{Z}}

\newcommand{\diam}{{\rm diam\,}}

\newcommand{\ve}{\varepsilon}

\renewcommand{\S}{\mathcal{S}}

\newcommand{\RCD}{\mathsf{RCD}}

\newcommand{\CD}{\mathsf{CD}}
\newcommand{\Geo}{{\rm Geo}}
\newcommand{\MCP}{\mathsf{MCP}}
\newcommand{\mm}{\mathfrak m}
\newcommand{\qq}{\mathfrak q}
\newcommand{\QQ}{\mathfrak Q}
\newcommand{\Tan}{{\rm Tan}}
\newcommand{\sfd}{\mathsf d}
\newcommand{\cP}{\mathcal P}

\newcommand{\Opt}{\mathrm{OptGeo}}
\DeclareMathOperator{\vol}{vol}

\newcommand{\cL}{\mathcal{L}}

\newcommand{\cR}{\mathcal{R}}
\newcommand{\cS}{\mathcal{S}}

\newcommand{\cH}{\mathcal{H}}

\newcommand{\ee} {{\rm e}}

\newcommand{\eps}{\varepsilon}
\newcommand{\heps}{\hat\varepsilon}
\newcommand{\beps}{\bar \varepsilon}
\def\co{\colon\thinspace}
\newcommand{\Ss}{\mathbb S}
\DeclareMathOperator{\intr}{int}
\DeclareMathOperator{\supp}{supp}

\newcommand{\cWR}{\mathcal{WR}}

\newcommand{\bRCD}{\boldsymbol\partial}

\begin{document}

\title[On the topology and the boundary of $N$-dimensional $\RCD(K,N)$ spaces]{ On the topology and the boundary of $N$-dimensional $\RCD(K,N)$ spaces }
\author{V. Kapovitch}  \thanks{V. Kapovitch: University of Toronto, email: vtk@math.toronto.edu. Supported in part by a Discovery grant from NSERC}

\author{A. Mondino}  \thanks{A.  Mondino: University of Oxford,  Mathematical Institut,  email:  Andrea.Mondino@maths.ox.ac.uk. Supported by the EPSRC First Grant  EP/R004730/1 ``Optimal transport and
geometric analysis''  and by the  ERC Starting Grant 802689 ``CURVATURE''}

\maketitle

\begin{abstract} 
 We establish topological regularity and stability of $N$-dimensional $\RCD(K,N)$ spaces (up to a small singular set), also called non-collapsed $\RCD(K,N)$ in the literature. We also introduce the notion of a boundary of such spaces and study its properties, including its behavior under Gromov-Hausdorff convergence.
\end{abstract}

\tableofcontents

\section{Introduction}
{
The notion of $\RCD(K,N)$ metric measure spaces (m.m.s.) was proposed and analyzed in \cite{gigli:laplacian, EKS, AMS} (see also \cite{CaMi}), as a finite
dimensional refinement of  $\RCD(K,\infty)$ m.m.s. which were first introduced  and studied in \cite{AGS} (see also \cite{AGMR12}).

 For  $K\in \R, N\in [1,\infty]$, the class of $\RCD(K,N)$ spaces is a subclass of $\CD(K,N)$ spaces pioneered by Lott-Villani  \cite{lottvillani:metric} and Sturm \cite{sturm:I, sturm:II} a few years earlier.  Roughly,  $\RCD(K,N)$ spaces are those $\CD(K,N)$ spaces where the Sobolev space $W^{1,2}(X,\sfd,\mm)$ is a Hilbert space (for a general $\CD(K,N)$ space, $W^{1,2}(X,\sfd,\mm)$ is only  Banach).
The motivation is that, while $\CD(K,N)$ spaces include Finsler manifolds, the class of $\RCD(K,N)$ spaces singles out the ``Riemannian''  $\CD(K,N)$ spaces.

 Both the classes of $\CD(K,N)$ and $\RCD(K,N)$ spaces are stable under pointed measured Gromov Hausdorff convergence  (pmGH for short), see \cite{lottvillani:metric, sturm:I, sturm:II, Vil, AGS, GMS2013}.
 Since the class of $\RCD(K,N)$ spaces includes the Riemannian manifolds with Ricci curvature bounded below by $K$ and dimension bounded above by $N$, the aforementioned stability  results imply that also their pmGH limits (the so-called Ricci-limits, throughly studied by Cheeger-Colding \cite{CC97, CC00a, CC00b}) are $\RCD(K,N)$.
 
 An interesting sub-class of Ricci-limits  already detected by Cheeger-Colding  \cite{CC97}, corresponds to the \emph{non-collapsed} ones. It consists of those Ricci-limits where the approximating sequence of smooth Riemannian manifolds have a uniform strictly positive lower bound on the volume of a unit ball. It follows from Colding's volume convergence \cite{Col97} that if $(X,\sfd,\mm)$ is a limit of $N$-manifolds with lower Ricci bound then it is non-collapsed if and only if $\mm=\cH^{N}$, and if and only if the Hausdorff dimension of $(X,\sfd)$ is $N$, see \cite{CC97}. The motivation for isolating the class of  \emph{non-collapsed} Ricci-limits is that they enjoy stronger structural properties than general (possibly collapsed) Ricci limits, for instance outside of a no-where dense set of measure zero they are topological manifolds.

 It is thus natural to consider those $\RCD(K,N)$ m.m.s. $(X,\sfd,\mm)$, where $\mm=\cH^{N}$, and call them ``non-collapsed''. The class of non-collapsed $\RCD(K,N)$ spaces has been the object of recent research by Kitabeppu \cite{Kitab-noncol}, De Philippis-Gigli \cite{GDPNonColl} (where the synthetic notion of non-collapsed $\RCD(K,N)$ spaces was formalized), Honda \cite{HondaNC}, Ketterer and the first author \cite{Kap-Ket-18}, Antonelli-Bru\'e-Semola \cite{Ant-Brue-Sem-19}.

 {

 \begin{remark}[Comparison between non-collapsed Ricci limits of Cheeger-Colding \cite{CC97} and non-collapsed $\RCD(K,N)$ spaces]
 The class of non-collapsed $\RCD(K,N)$ spaces strictly contains the non-collapsed Ricci limits of Cheeger-Colding \cite{CC97}. Indeed:
 \begin{enumerate}
 \item  \cite{CC97} considered sequences of manifolds without boundary and proved that in the non-collapsing situation the limit does not have boundary. More precisely, in the terminology we introduce here,  \cite{CC97}  prove that the the limit does not have reduced boundary, and it follows from  Theorem \ref{thm-no-bry-mainIntro} that more generally  the limit does not have $\RCD$-boundary either.
 In particular a convex body in $\R^{N}$ with boundary cannot arise as a non-collapsed Ricci limit of manifolds without boundary; however this is a non-collapsed $\RCD(0,N)$ space. 
 \\Although there are results about Gromov Hausdorff (pre)-compactness of $N$-manifolds  with Ricci bounded below and with boundary satisfying suitable conditions (e.g. bounded second fundamental form  \cite{Wong}), extending Cheeger-Colding theory to the corresponding limit spaces with boundary seems to be not yet addressed in the literature. 
 \\The theory of $\RCD(K,N)$ spaces, and this paper in particular, should be useful in this regard. Indeed a Riemannian $N$-manifold $(M,g)$  with Ricci bounded below by $K$ and with convex boundary (i.e. $II_{\partial M}\geq 0$) is a (non-collapsed) $\RCD(K,N)$ space  \cite[Theorem 2.4]{HanMeasRig}.
 Thanks to the stability of (resp. non-collapsed) $\RCD(K,N)$ spaces, it follows that the (resp. non-collapsed) pmGH limits of such objects are (resp. non-collapsed) $\RCD(K,N)$  as well.
  \item By  \cite{palvs} or ~\cite{Ket} { or ~\cite{GGKMS}}, a cone (resp. a spherical suspension) over $\mathbb{RP}^{2}$ is an example of a non-collapsed $\RCD(0,3)$ (resp. $\RCD(1,3)$ space). It was noted in a discussion between De Philippis-Mondino-Topping that such spaces cannot arise as non-collapsed Ricci limits. Indeed on the one hand they are not topological manifolds, and on the other hand it was proved in \cite{Sim, SimTop} that non-collapsed 3-dimensional Ricci limits  are topological manifolds (see also Remark \ref{rem:CoDim3Top}).
 \end{enumerate}
 \end{remark}
}

In this note we study topological and rigidity properties of non-collapsed $\RCD$ spaces.
As in \cite{CC97}, we will adopt the notation that $\R_{+}\times \R^k\ni (\varepsilon,x)\mapsto  \Psi(\varepsilon|x)\in \R_{+}$ denotes a non-negative function satisfying that, for any fixed $x=(x_1,\ldots,x_k)$, $\lim_{\varepsilon\to 0} \Psi(\varepsilon|x)=0$.
 We prove the following results.
}

\begin{theorem}\label{thm:ConvergenceMan}
Fix some $K\in \R$ and $N\in \N$. Let $\{(X_{i},\sfd_{i}, \cH^{N}, \bar{x}_{i})\}_{i\in \N}$ be a sequence of $\RCD(K,N)$ spaces pmGH-converging to a {closed} smooth Riemannian manifold $(M,g)$ { of dimension $N$}.
Then there exists $i_{0}\in \N$ such that for all $i\geq i_{0}$ the space $(X_{i},\sfd_{i})$ is  homeomorphic to $(M,g)$, in particular $(X_{i},\sfd_{i})$ is a topological manifold.

\end{theorem}

\begin{theorem}\label{Thm:ConeInf}
Fix some  $N\in \N$. Let $(X,\sfd, \cH^{N})$ be an  $\RCD(0,N)$ space.
\\If $X$ admits a tangent space at infinity isometric  to $\R^{N}$ then $(X,\sfd, \cH^{N})$ is isomorphic to $\R^{N}$ as a m.m.s.. 
\\More precisely the following holds.  For every $N\in \N, N\geq 1$ there exists $\varepsilon(N)>0$ with the following property. 
\begin{itemize}
\item If there exists $x\in X, R_{0}\geq 0$ such that $\sfd_{GH}(B_{r}(x),B_{r}(0^{N}))\leq \varepsilon r$ for all $r\geq R_{0}$,  then $X$ is homeomorphic to $\R^{N}$ and $\cH^{N}(B_{r}(x))\geq (1-\Psi(\varepsilon|N))  \omega_{N} r^{N}$ for all $r\geq R_{0}$;
\item If there exists $x\in X, R_{0}\geq 0$ such that $\cH^{N}(B_{r}(x))\geq (1-\varepsilon)  \omega_{N} r^{N}$ then $X$ is homeomorphic to $\R^{N}$ and $\sfd_{GH}(B_{r}(x),B_{r}(0^{N}))\leq \Psi(\varepsilon|N)  r$ for all $r\geq R_{0}$.
\end{itemize}
\end{theorem}

\begin{theorem}[Sphere Theorem]\label{thm:sphere}
For every $N\in \N, N\geq 1$ there exists $\varepsilon(N)>0$ with the following property. 
\\Let $(X,\sfd, \cH^{N})$ be an $\RCD(N-1,N)$ space for some  $N\in \N, N\geq 1$.
\begin{itemize}
\item If $\sfd_{GH}(X,\bS^{N})\leq \varepsilon$ then $X$ is homeomorphic to $\bS^{N}$ and $\cH^{N}(X)\geq (1-\Psi(\varepsilon|N))  \cH^{N}(\bS^{N})$;
\item If $\cH^{N}(X)\geq (1-\varepsilon) \cH^{N}(\bS^{N})$ then  $X$ is homeomorphic to $\bS^{N}$ and $\sfd_{GH}(X,\bS^{N})\leq \Psi(\varepsilon|N)$.
\end{itemize}
\end{theorem}

We also introduce the notion of a boundary of a non-collapsed $\RCD$ space and establish its various properties. In particular we prove some stability results about behavior of the boundary under limits.

{
The definition of a boundary point is inductive on the dimension of the space. Roughly, 
\begin{align*}
&\text{The \emph{$\RCD$-boundary} $\bRCD X$ of a non-collapsed $\RCD(K,N)$ space $(X,\sfd, \cH^{N})$}\\
&\text{ consists of points admitting a tangent space with boundary}.
\end{align*}
For the precise notions see Definition \ref{def:bRCD}.
\\  From \cite{Kitab-noncol, GDPNonColl} every tangent space to a non-collapsed $\RCD(K,N)$ space is a metric measure cone $C(Z)$ over a non-collapsed $\RCD(N-2,N-1)$ space $Z$ (see Lemma \ref{lem:TanncRCD}). Since the cone $C(Z)$ has boundary if and only if $Z$ has, the induction on the dimension is clear and stops in dimension one where any non-collapsed $\RCD(0,1)$ space is isomorphic to either a line, a circle, a half line or a closed interval  \cite{KL}; by definition we say that the latter two have boundary and the former two don't.

 A somewhat different notion of a boundary of noncollapsed $\RCD$ spaces was proposed in ~\cite{GDPNonColl}. In our notation,   De Philippis and Gigli proposed to call the boundary of $X$ the closure of what we call \emph{reduced} boundary of $X$ where by reduced boundary we mean  $\S^{N-1}\setminus \S^{N-2}$ (see Definition \ref{def:bRCD} ). It's possible that these two definitions are equivalent but this is not clear at the moment. In fact, it's not clear if either of these sets is contained in the other. We do show that the reduced boundary is a subset of the boundary (Lemma~\ref{lem:TanSpXNC}).

 Our notion of boundary agrees with the one of boundary of an Alexandrov space i.e. in case $(X,\sfd)$ is a finite dimensional Alexandrov space with curvature bounded below.  Moreover, for Alexandrov spaces our notion of the boundary is known to coincide with the one suggested by De Philippis and Gigli~\cite{Per-stab}.   
 
 Our notion of boundary is compatible with the  topological boundary in case $X$ is a topological manifold in the following sense:

\begin{theorem}[Corollary \ref{manifold-no-boundary} { and Corollary \ref{cor:bilip}}] \label{manifold-no-boundaryIntro}
Let $(X,\sfd, \cH^{N})$ be a non-collapsed $\RCD(K,N)$ space.  If $(X,\sfd)$ is a  topological $N$-manifold with boundary $\partial X$, then  $\boldsymbol\partial X \subset \partial X$.\\In particular if   $X$ is a topological $N$-manifold without boundary in manifold sense, then it is also without boundary in the $\RCD$ sense.
\\  Moreover, if $(X,\sfd)$ is bi-Lipschitz homeomorphic to a smooth $N$-manifold with boundary  then also the reverse inclusion holds, namely  $\boldsymbol\partial X = \partial X$.
\end{theorem}

 We are grateful to Alexander Lytchak for pointing out to us that the above theorem can be improved to show the following stronger result:

\begin{theorem}\label{thm:bilip-preserve-bry} (Proposition ~\ref{bilip-preserve-bry})
Let $(X,\sfd_{X},\cH^{N})$ and  $(Y,\sfd_{Y},\cH^{N})$ be non-collapsed  $\RCD(K,N)$ spaces. Assume there is a bi-Lipschitz homeomorphism $f:X\to Y$. Then $f(\bRCD X)=\bRCD Y$.

\end{theorem}

One of the main results of the paper is the following structure theorem of non-collapsed $\RCD(K,N)$ spaces, see Theorem \ref{thm-structure} for a more precise statement. It should  be compared with the structure theory of non-collapsed Ricci-limit spaces of Cheeger-Colding \cite{CC97}; a major difference is that while it is known (still from \cite{CC97}) that non-collapsed Ricci limits do not have boundary, in general a non-collapsed $\RCD(K,N)$ space does have boundary (e.g. a  closed convex subset of $\R^{N}$ with nonempty interior). 

\begin{theorem}\label{thm-structureIntro}
Let $(X,\sfd,\cH^{N})$ be a non-collapsed $\RCD(K,N)$ space for some $K\in \R, N\in \N$.
\\Then $\bRCD X$ has Hausdorff dimension at most $N-1$.
\\Moreover  there exists an open subset $M\subset X$ bi-H\"older homeomorphic to a smooth  manifold, such that
$$
X=M \, \mathring\cup \,  \bRCD X \, \mathring\cup  \, (X\setminus (\bRCD X\cup M)),
$$
where $X\setminus (\bRCD X\cup M)$ has Hausdorff dimension at most $N-2$. 
\\In words:   $X$  is the disjoint union of a manifold part of dimension $N$, a boundary part of Hausdorff dimension at most $N-1$, and a singular set of Hausdorff dimension at most $N-2$. 

Moreover, if $\cH^{N-1}(\bRCD X)=0$, then $M$ is path connected  and the induced inner metric on $M$ coincides with the restriction of the ambient metric $\sfd$.
\end{theorem}

We suspect that the sharp codimension for the ``topologically singular set'' is three:
\begin{conjecture}
 Let $(X,\sfd,\cH^{N})$ be a non-collapsed $\RCD(K,N)$ space. Then there exists an open subset $M\subset X$ such that $M$ is homeomorphic to a topological manifold,  ${\mathcal R}_{N}\subset M$ and $X\setminus(\bRCD X \cup M)$ has codimension at least 3.
\end{conjecture}

\begin{remark}\label{rem:CoDim3Top}
Co-dimension 3 would be sharp in general: indeed $C(\RP^{2})$  is a non-collapsed $\RCD(0,3)$ space and the vertex of the cone is not a manifold point.

Notice the difference with the case of non-collapsed Ricci limits, where a conjecture by Anderson-Cheeger-Colding-Tian states that the space is a manifold out of a singular set of co-dimension $4$. This was proved for 3-dimensional compact Ricci limits by M. Simon \cite{Sim} under a global non-collapsing assumption, for general non-compact non-collapsed 3-dimensional Ricci limits by M. Simon-Topping \cite{SimTop}, and for non-collapsed Ricci-double sided limits of arbitrary dimension by Cheeger-Naber~\cite{Ch-Na-codim4}.
\end{remark}

In Section \ref{Sec:BoundConv} we prove several results about the behaviour of the boundary under pointed Gromov Hausdorff convergence, here we just state the following  which { (sharpens and)} generalizes  to non-collapsed $\RCD(K,N)$ spaces   \cite[Theorem 6.1]{CC97}  by Cheeger-Colding.

\begin{theorem}[Theorem \ref{thm-no-bry-main}]\label{thm-no-bry-mainIntro} 
Let \{$(X_{i}, \sfd_{i}, \cH^{N}_{\sfd_{i}})\}_{i\in \N}$ be a sequence of non-collapsed $\RCD(K,N)$ spaces. Assume that
\begin{itemize}
\item $\{(X_{i}, \sfd_{i}, p_{i})\}_{i\in \N}$ converge to  $(X, \sfd, p)$ in pointed Gromov-Hausdorff sense.
\item $\cH^{N}_{\sfd_{i}}(B_1(p_i))\ge v>0$ for all $i$.
\item Each $(X_{i},\sfd_{i})$ is a topological manifold without boundary.
\end{itemize}
Then $(X, \sfd, \cH^{N})$ is a non-collapsed $\RCD(K,N)$ space with  $\bRCD X =\emptyset$.
\end{theorem}
 Note that Theorem \ref{manifold-no-boundaryIntro} follows directly from a localized version of Theorem \ref{thm-no-bry-mainIntro}, namely Theorem \ref{thm-no-bry-main}. 
 \\

In Section \ref{Sec:SeqOpenNCRCD}, we investigate sequences of $\RCD(K,N)$ spaces where the limit is non-collapsed $\RCD(K,N)$ (or more generally weakly non-collapsed, i.e. $\mm\ll\cH^{N}$). We will  prove  several results stating roughly that if the limit  of a pmGH sequence of $\RCD(K,N)$ spaces is  (weakly) non-collapsed, then the same is true  eventually for the elements of the sequence; thus establishing a sort of ``sequential openness'' of this class of spaces. For the precise statements, see Theorem \ref{thm:Stabwnc} and Theorem \ref{thm:Stabnc}; here we only mention a very special case when the limit is a smooth Riemannian manifold.

\begin{theorem}[Theorem \ref{thm:XiMman}]\label{thm:XiMmanIntro} 
Let $(M,g)$ be a compact Riemannian manifold of dimension $N$. Let $\{(X_{i}, \sfd_{i},\mm_{i}, \bar{x_{i}})\}_{i\in \N}$ be a sequence of $\RCD(K,N)$ spaces, for some $K\in \R$,  converging to 
$(M,g)$ in pointed measured Gromov Hausdorff sense. Then there exists $i_{0}\in \N$ such that  for all $i\ge i_0$ it holds that
\begin{itemize}
\item $(X_{i}, \sfd_{i},\mm_{i})$ is a compact non-collapsed $\RCD(K,N)$: more precisely, there exists a sequence $c_{i}\to 1$ such that $\mm_{i}=c_{i} \cH^{N}$.
\item $(X_{i}, \sfd_{i})$ is homeomorphic to $M$ via a bi-H\"older homeomorphism.
\end{itemize}
\end{theorem}

{ 
As a natural application of some of the main results of the paper, we next present an almost-rigidity result on the spectrum of the Laplace operator. In the proof we will use the  Sphere Theorem \ref{thm:sphere}, Theorem \ref {thm:XiMmanIntro},  the stability of the spectrum for $\RCD(N-1,N)$ spaces \cite[Theorem 7.8]{GMS2013} and the higher order Obata's  rigidity result \cite[Theorem 1.4]{ketterer3}. { The complete proof can be found at the end of Section \ref{Sec:SeqOpenNCRCD}.}

\begin{corollary}\label{eq:lambda1N=N}
For every  $N\in \N, N\geq 2$, and every  $\varepsilon>0$ there exists $\delta=\Psi(\varepsilon|N)$ with the following property.
If $(X,\sfd,\mm)$ is an $\RCD(N-1,N)$ space such that $N\leq \lambda_{1}\leq \lambda_{2}\leq \ldots \leq \lambda_{N+1} \leq N+\delta$, then 
\begin{enumerate}
\item  $(X,\sfd,\mm)$ is a non-collapsed  $\RCD(N-1,N)$ space; more precisely there exists $c>0$ such that $\mm=c\cH^{N}$.
\item  $(X,\sfd)$ is homeomorphic to ${\mathbb S}^{N}$.
\item  $\cH^{N}(X)\in \left[(1-\varepsilon) \cH^{N}(\mathbb S^{N}),  \cH^{N}(\mathbb S^{N}) \right]$.
\item $\sfd_{GH} (X,{\mathbb S}^{N})\leq \varepsilon.$
\end{enumerate}
\end{corollary}

In case $(X,\sfd,\mm)$ is a smooth Riemannian $N$-manifold with Ricci $\geq N-1$, Corollary \ref{eq:lambda1N=N} is a consequence of \cite[Theorem 1.1]{PetEigenvalue}\cite{Col96b, Col96a} (see also \cite{BertrandEigenvalue}).
With a similar compactness-contradiction argument used to prove Corollary \ref{eq:lambda1N=N}, one can obtain the following almost version of the  Erbar-Sturm's rigidity result  \cite[Corollary 1.4]{ErSt18}. 
 \begin{corollary}[Almost version of Erbar-Sturm's Rigidity]\label{cor:ESAR}
For every  $N\in \N, N\geq 2$, and every  $\varepsilon>0$ there exists $\delta=\Psi(\varepsilon|N)$ with the following property.
If $(X,\sfd,\mm)$ is an $\RCD(N-1,N)$ space such that 
$$\left|\int_{X}\int_{X} \cos \sfd(x,y)\,  \mm(dx) \, \mm(dy) \right|<\delta,$$
 then the same conclusions {\rm (1)-(4)} of Corollary \ref{eq:lambda1N=N} hold.
\end{corollary}
Let us stress that Corollary \ref{cor:ESAR} seems new even in case when $(X,\sfd,\mm)$ is a smooth Riemannian $N$-manifold with Ricci $\geq N-1$.

A few days after the present work was posted on arXiv, an independent paper by Honda and Mondello \cite{HM} appeared, containing  similar versions of  Theorem \ref{thm:sphere} and Corollary \ref{eq:lambda1N=N}.
The two papers are quite different in scope though: while we are mainly concerned with the topology and the boundary of non-collapsed
$\RCD$ spaces,   \cite{HM} is more focused on sphere theorems and their application to stratified spaces.

\subsubsection*{Acknowledgements}   
The project started  while both authors were in residence at the Mathematical Sciences Research Institute in Berkeley, California during the Spring 2016 semester, supported by the National Science Foundation under Grant No. DMS-1440140. We thank the organizers of the Differential Geometry Program and MSRI for providing great environment for research and collaboration. We are grateful to Misha Kapovich for suggesting  the proof of Theorem~\ref{thn-comp-exhaustion}, { to Fabio Cavalletti and to the anonymous referee for carefully reading the paper and their comments which improved the exposition. }
\\Vitali Kapovitch is supported in part by a Discovery grant from NSERC.
\\Andrea Mondino acknowledges the support of   the EPSRC First Grant  EP/R004730/1 and of the ERC Starting Grant 802689. 

}

\section{Preliminaries}
{
Throughout the paper a metric measure space (m.m.s. for short) is a  triple $(X,\sfd,\mm)$ where $(X,\sfd)$ is a complete, proper and separable metric space and $\mm$ is a non-negative Borel measure on $X$ finite on bounded subsets and satisfying $\supp(\mm)=X$.  The properness assumption is motivated by the synthetic Ricci curvature lower bounds/dimensional upper bounds we will assume to hold. 

For $k>0$ we will denote  by $\cH^{k}$ the $k$-dimensional Hausdorff measure of $(X,\sfd)$.  The Gromov-Hausdorff distance between two metric spaces is denoted by $\sfd_{GH}$.

\subsection{Ricci curvature lower bounds and dimensional upper bounds for  metric measure spaces}\label{SS:RCDDef}

We denote by 
$$
\Geo(X) : = \{ \gamma \in C([0,1], X):  \sfd(\gamma_{s},\gamma_{t}) = |s-t| \sfd(\gamma_{0},\gamma_{1}), \text{ for every } s,t \in [0,1] \}
$$
the space of constant speed geodesics. The metric space $(X,\sfd)$ is a \emph{geodesic space} if and only if for each $x,y \in X$ 
there exists $\gamma \in \Geo(X)$ so that $\gamma_{0} =x, \gamma_{1} = y$. Given two points $x,y$ in  a geodesic metric space $(X, \sfd)$ we will denote by $[x,y]$ a shortest geodesic between $x$ and $y$.
\\Recall that, for complete geodesic spaces, local compactness is equivalent to properness (a metric space is proper if every closed ball is compact).

\medskip

We denote with  $\mathcal{P}(X)$ the  space of all Borel probability measures over $X$ and with  $\mathcal{P}_{2}(X)$ the space of probability measures with finite second moment.
$\mathcal{P}_{2}(X)$ can be  endowed with the $L^{2}$-Kantorovich-Wasserstein distance  $W_{2}$ defined as follows:  for $\mu_0,\mu_1 \in \mathcal{P}_{2}(X)$,  set
\begin{equation}\label{eq:W2def}
  W_2^2(\mu_0,\mu_1) := \inf_{ \pi} \int_{X\times X} \sfd^2(x,y) \, \pi(dxdy),
\end{equation}
where the infimum is taken over all $\pi \in \mathcal{P}(X \times X)$ with $\mu_0$ and $\mu_1$ as the first and the second marginal.
The space $(X,\sfd)$ is geodesic  if and only if the space $(\mathcal{P}_2(X), W_2)$ is geodesic. 
\\We will also consider the subspace $\mathcal{P}_{2}(X,\sfd,\mm)\subset \mathcal{P}_{2}(X)$
formed by all those measures absolutely continuous with respect with $\mm$.
\medskip

In order to formulate curvature properties for $(X,\sfd,\mm)$  we recall the definition of the  distortion coefficients:  for $K\in \R, N\in [1,\infty), \theta \in (0,\infty), t\in [0,1]$, set 
\begin{equation}\label{eq:deftau}
\tau_{K,N}^{(t)}(\theta): = t^{1/N} \sigma_{K,N-1}^{(t)}(\theta)^{(N-1)/N},
\end{equation}
where the $\sigma$-coefficients are defined 
as follows:
given two numbers $K,N\in \R$ with $N\geq0$, we set for $(t,\theta) \in[0,1] \times \R_{+}$, 
\begin{equation}\label{eq:Defsigma}
\sigma_{K,N}^{(t)}(\theta):= 
\begin{cases}
\infty, & \textrm{if}\ K\theta^{2} \geq N\pi^{2}, \crcr
\displaystyle  \frac{\sin(t\theta\sqrt{K/N})}{\sin(\theta\sqrt{K/N})} & \textrm{if}\ 0< K\theta^{2} <  N\pi^{2}, \crcr
t & \textrm{if}\ K \theta^{2}<0 \ \textrm{and}\ N=0, \ \textrm{or  if}\ K \theta^{2}=0,  \crcr
\displaystyle   \frac{\sinh(t\theta\sqrt{-K/N})}{\sinh(\theta\sqrt{-K/N})} & \textrm{if}\ K\theta^{2} \leq 0 \ \textrm{and}\ N>0.
\end{cases}
\end{equation}

\noindent
Let us also recall the definition of the R\'enyi Entropy functional ${\mathcal E}_{N} : \cP(X) \to [0, \infty]$, 
\begin{equation}\label{eq:defEnt}
{\mathcal E}_{N}(\mu)  := \int_{X} \rho^{1-1/N}(x) \,\mm(dx),
\end{equation}
where $\mu = \rho \mm + \mu^{s}$ with $\mu^{s}\perp \mm$.
\medskip

\begin{definition}[$\CD$ condition]\label{def:CD}
Let $K \in \R$ and $N \in [1,\infty)$. A metric measure space  $(X,\sfd,\mm)$ verifies $\CD(K,N)$ if for any two $\mu_{0},\mu_{1} \in \cP_{2}(X,\sfd,\mm)$ 
with bounded support there exist a $W_{2}$-geodesic $(\mu_{t})_{t\in [0,1]}\subset \cP_{2}(X,\sfd,\mm)$ and  $\pi\in \cP(X\times X)$ $W_{2}$-optimal plan, such that for any $N'\geq N, t\in [0,1]$:
\begin{equation}\label{eq:defCD}
{\mathcal E}_{N'}(\mu_{t}) \geq \int \tau_{K,N'}^{(1-t)} (\sfd(x,y)) \rho_{0}^{-1/N'} 
+ \tau_{K,N'}^{(t)} (\sfd(x,y)) \rho_{1}^{-1/N'} \,\pi(dxdy).
\end{equation}
\end{definition}

Throughout this paper, we will always assume the proper metric measure space $(X,\sfd,\mm)$ to satisfy $\CD(K,N)$, for some $K,N \in \R$. This will imply in particular that $(X,\sfd)$ is geodesic.

A variant of the $\CD$ condition, called  reduced curvature dimension condition and denoted by  $\CD^{*}(K,N)$ \cite{BS10},  asks for the same inequality \eqref{eq:defCD} of $\CD(K,N)$ but  the
coefficients $\tau_{K,N}^{(t)}(\sfd(\gamma_{0},\gamma_{1}))$ and $\tau_{K,N}^{(1-t)}(\sfd(\gamma_{0},\gamma_{1}))$ 
are replaced by $\sigma_{K,N}^{(t)}(\sfd(\gamma_{0},\gamma_{1}))$ and $\sigma_{K,N}^{(1-t)}(\sfd(\gamma_{0},\gamma_{1}))$, respectively.
In general, the $\CD^{*}(K,N)$ condition is weaker than  $\CD(K,N)$. 
\medskip

On $\CD(K,N)$ spaces,  the  classical Bishop-Gromov volume growth estimate holds. In order to state it, for a  fixed a point $x_{0}\in X$ let 
$$ v(r)= \mm(\bar{B}_{r}(x_{0})) $$
be the volume of the closed metric ball of radius $r$ and center $x_{0}$ and let 
$$
s(r):=\limsup_{\delta\to 0} \frac{1}{\delta} \mm(\bar{B}_{r+\delta}(x_{0}) \setminus  B_{r}(x_{0}))
$$
be the codimension one volume of the corresponding spheres.

\begin{theorem}[Bishop-Gromov in $\CD(K,N)$, Theorem 2.3 \cite{sturm:II}]\label{thm:BishopGromov}
Let $(X,\sfd,\mm)$ be a  $\CD(K,N)$ space for some $K,N\in \R, N> 1$. Then for all $x_{0}\in X$ and all $0< r< R\leq \pi \sqrt{N-1/(K \vee 0)}$ it holds:

\begin{equation}\label{eq:BSspheres}
\frac{s(r)}{s(R)} \geq 
\begin{cases}
\left( \frac{\sin(r \sqrt{K/(N-1)}}{ \sin(R \sqrt{K/(N-1)} }  \right)^{N-1},& \quad \text{if } K>0,\medskip \\
\left( \frac r  R  \right)^{N-1},& \quad \text{if } K=0, \medskip \\
\left( \frac{\sinh(r \sqrt{K/(N-1)}}{ \sinh(R \sqrt{K/(N-1)} }  \right)^{N-1},& \quad \text{if } K<0,
\end{cases}
\end{equation}
and 
\begin{equation}\label{eq:BSballs}
\frac{v(r)}{v(R)} \geq { v_{K,N}(r):=}
\begin{cases}
\frac{\int_{0}^{r} \left(\sin(t \sqrt{K/(N-1)})\right)^{N-1} dt}{ \int_{0}^{R} \left(\sin(t \sqrt{K/(N-1)}) \right)^{N-1} dt} ,& \quad \text{if } K>0, \medskip \\
\left(\frac r  R  \right)^{N},& \quad \text{if } K=0, \medskip \\
\frac{\int_{0}^{r} \left(\sinh(t \sqrt{K/(N-1)})\right)^{N-1} dt}{ \int_{0}^{R} \left(\sinh(t \sqrt{K/(N-1)}) \right)^{N-1} dt} ,& \quad \text{if } K<0. 
\end{cases}
\end{equation}

\end{theorem}

One crucial property of the $\CD(K,N), \CD^*(K,N)$ conditions is the stability under measured Gromov Hausdorff convergence of m.m.s., so that Ricci limit spaces are $\CD(K,N)$. Moreover, on the one hand it is possible to see that Finsler manifolds are allowed as $\CD(K,N)$-space while on the other hand from the work of Cheeger-Colding \cite{CC97,CC00a,CC00b}  it was understood that purely Finsler structures never appear as Ricci limit spaces.   Inspired by this fact, in \cite{AGS}, Ambrosio-Gigli-Savar\'e proposed a  strengthening of the $\CD$ condition in order to enforce, in some weak sense, a Riemannian-like behavior of spaces with a curvature-dimension bound (to be precise in  \cite{AGS} it was analyzed the case of strong-$\CD(K,\infty)$ spaces endowed with a probability reference measure $\mm$; the axiomatization has been then simplified and generalized in  \cite{AGMR12} to allow $\CD(K,\infty)$-spaces endowed with a $\sigma$-finite reference measure).
The finite dimensional refinement $\RCD^{*}(K,N)$ with $N<\infty$  has been subsequently proposed and extensively studied
in \cite{gigli:laplacian, EKS,AMS}.  Such a strengthening consists in requiring that the space $(X,\sfd,\mm)$ is such that the Sobolev space $W^{1,2}(X,\sfd,\mm)$ is Hilbert (or, equivalently, the heat flow is linear; or, equivalently, the Laplacian is linear), a condition named ``infinitesimal Hilbernian'' in \cite{gigli:laplacian}. 
\\More precisely,  on a m.m.s. there is a canonical notion of ``modulus of  the differential of a function'' $f$, called weak upper differential and denoted with $|Df|_w$; with this object one defines the Cheeger energy
$${\rm Ch}(f):=\frac 1 2 \int_X |Df|_w^2 \,  \mm.$$
The  Sobolev space $W^{1,2}(X,\sfd,\mm)$ is by definition  the space of $L^2(X,\mm)$ functions having finite Cheeger energy, and it is endowed with the natural norm  $\|f\|^2_{W^{1,2}}:=\|f\|^2_{L^2}+2 {\rm Ch}(f)$ which makes it a Banach space. We remark that, in general, $W^{1,2}(X,\sfd,\mm)$ is not Hilbert (for instance, on a smooth Finsler manifold the space $W^{1,2}$ is Hilbert if and only if the manifold is actually Riemannian); in case  $W^{1,2}(X,\sfd,\mm)$ is  Hilbert then we say that $(X,\sfd,\mm)$ is \emph{infinitesimally Hilbertian}. 
\\ We refer to the aforementioned papers and references therein for a general account
on the synthetic formulation of the latter Riemannian-type Ricci curvature lower bounds; for a survey of results, see the Bourbaki seminar \cite{VilB} and the recent ICM-Proceeding \cite{AmbrosioICM}.

A key property of $\RCD^{*}(K,N)$ is  stability under pointed measured Gromov Hausdorff convergence  \cite{AGS,GMS2013} so that Ricci limit spaces are $\RCD^{*}(K,N)$ spaces.

To conclude we recall also that recently  Cavalletti and E. Milman \cite{CaMi} proved the equivalence
of $\CD(K,N)$ and $\CD^{*}(K,N)$, together with the local-to-global property for $\CD(K,N)$, in the framework of  essentially non-branching m.m.s. having $\mm(X) < \infty$.

It is worth also mentioning that a m.m.s. verifying $\RCD^{*}(K,N)$ is essentially non-branching (see \cite[Corollary 1.2]{rajalasturm})
implying the equivalence of $\RCD^{*}(K,N)$ and  $\RCD(K,N)$, in case of finite total measure.

\begin{remark}
The results in \cite{CaMi} are stated for spaces with finite reference measure but the kind of
arguments used seems to indicate that the same also holds without this restriction. For
this reason (and also to uniformize our notation with \cite{GDPNonColl} where the non-collapsed $\RCD(K,N)$ spaces have been formalized), in the present  paper we shall work with
$$\RCD(K,N):=\CD(K,N)+ \text{Infin. Hilbertian}$$
spaces, rather than with 
$$\RCD^{*}(K,N):=\CD^{*}(K,N)+ \text{Infin. Hilbertian}$$
ones, which have been popular in the last years. In any case, all our arguments are local in nature and since the
local versions of $\CD(K,N)$ and $\CD^{*}(K,N)$ are known to be equivalent from the original paper \cite{BS10}, our results are independent by the global equivalence of $\RCD(K,N)$ vs. $\RCD^{*}(K,N)$. Actually our proofs can be carried without modification directly for $\RCD^{*}(K,N)$ spaces.
\end{remark}

Following the terminology of \cite{GDPNonColl}  (motivated by \cite{CC97}), we say that the $\RCD(K,N)$ space $(X,\sfd,\mm)$ is \emph{non-collapsed} if $\mm=\cH^{N}$, the $N$-dimensional Hausdorff measure. 
We also say that the $\RCD(K,N)$ space $(X,\sfd,\mm)$ is \emph{weakly non-collapsed} if $\mm\ll \cH^{N}$. It was recently proved  \cite{HondaNC} that a compact weakly non-collapsed $\RCD(K,N)$ space is actually non-collapsed (up to a constant rescaling of the measure). This is expected to be true in the noncompact case as well.

 }

\subsection{Regular and singular sets}

Given a complete metric space $(X,\sfd)$, $N\in \N$ with $N\geq 1$, and  $\varepsilon, r>0$, denote
\begin{equation}\label{eq:defRNepsr}
(\cR_{N})_{\varepsilon, r}:=\{x\in X\,:\, \exists t>r \text{ such that } \sfd_{GH}(B_{s}^{X}(x), B_{s}^{\R^{N}}(0))\leq \varepsilon s, \text{ for all }s\in(0,t] \}.
\end{equation}
The $(\varepsilon,N)$-regular set  $(\cR_{N})_{\varepsilon}$ of $(X,\sfd)$ is defined by 
\begin{equation}\label{eq:defRNeps}
(\cR_{N})_{\varepsilon}:=\cup_{r>0} (\cR_{N})_{\varepsilon, r}.
\end{equation}
In turn, the  $N$-regular set  $\cR_{N}$ of $(X,\sfd)$ is defined by $\cR_{N}:=\cap_{\varepsilon>0} (\cR_{N})_{\varepsilon}$.
\\{ It follows from \cite{MN} that 
\begin{equation}\label{eq:RNfullMeas}
\mm(X\setminus \cR_{N})=0 \quad \text{ if  $(X,\sfd,\mm)$ is weakly non-collapsed $\RCD(K,N)$.}
\end{equation}}
Notice that, if $(X,\sfd,\mm)$  is an $\RCD(K,N)$ space,  the monotonicity in Bishop-Gromov ensures that for every $x\in X$ the following (possibly infinite)  limit exists 
\begin{equation}\label{eq:defthetaN}
\vartheta_{N}[(X,\sfd,\mm)](x)=\vartheta_{N}(x)=\lim_{r\to 0} \frac{\mm(B_{r}(x))}{\omega_{N} r^{N}},
\end{equation}
{and that the function $x\mapsto \vartheta_{N}(x)$ is lower semicontinuous.}
It was  proved in \cite{GDPNonColl} that 
\begin{equation}\label{noncoll-density-1}
\text{An $\RCD(K,N)$ space $(X, \sfd , \mm)$ is non-collapsed if
$\vartheta_N=1$ $\mm$-a.e..} 
\end{equation}
Moreover, from ~\cite[Corollary 1.7]{GDPNonColl}:
\begin{equation}\label{eq:RegSetDens}
\begin{split}
 &\text{ If $(X,\sfd,\cH^{N})$ is a non-collapsed $\RCD(K,N)$ space } \\
& \text{then  $\vartheta_{N}(x) \leq 1$ for every $x\in X$ and  $\vartheta_N(x) = 1$ if and only if $x\in {\mathcal R}_{N}$.}
 \end{split}
\end{equation} 

 Given a metric space $(Z,\sfd_{Z})$, we define the metric cone $C(Z)$  over $Z$ to be the completion of $\R^{+}\times Z$ endowed with the metric
\begin{equation}\label{eq:defConeDist}
\sfd_{C}((r_{1},z_{1}), (r_{2},z_{2}))^{2}= 
\begin{cases}
r_{1}^{2}+r_{2}^{2}-2 r_{1}r_{2} \cos(\sfd_{Z}(z_{1},z_{2})), &\quad \text{ if } \sfd_{Z}(z_{1},z_{2})\leq \pi \\
(r_{1}+r_{2})^{2}, & \quad \text{ if } \sfd_{Z}(z_{1},z_{2})\geq \pi.
\end{cases}
\end{equation}
If $(Z,\sfd_{Z},\mm_{Z})$ is a m.m.s. the metric cone  $C(Z)$ can be endowed with a family of natural cone measures $\mm_{C,N}$, depending on a real parameter $N>1$ playing a dimensional role, as 
\begin{equation}\label{eq:ConeMeas}
\mm_{C,N}=t^{N-1} dt \otimes \mm_{Z}. 
\end{equation}

In order to make the notation short, when there is no ambiguity on the metric or on the measure, we will simply write $X$ for the metric space $(X,\sfd)$ (resp. for the m.m.s. $(X,\sfd,\mm)$).
\\We adopt the following quantitative stratification notations and terminology from ~\cite{cheeger-jiang-naber-18}.
\begin{definition} $\quad$
\begin{itemize}
\item A metric space $X$ is called \emph{$k$-symmetric} if it is isometric to $\R^k\times C(Z)$  for some metric space $Z$. 
\item Given $\eps >0$ we say that a ball $B_r(x)\subset X$ is $(k,\eps)$-symmetric if there is a $k$-symmetric space $X'=\R^k\times C(Z)$ such that $\sfd_{GH}(B_r(x), B_r(x'))<\eps r$ where $x'$ is the vertex of the cone in $X'$.
\item  Given $\eps, r>0, k\in \N$ we define $\S_{\eps,r}^k(X)$ to be the set of points $p$ in $X$ such that $B_s(p)$ is not $(k+1,\eps)$-symmetric for any $r\le s\le 1$.
\item Lastly, we define $\S_\eps^k(X)$ as $\cap_{r>0}\S_{\eps,r}^k(X)$.
\end{itemize}
When the space $X$ in question is clear we will often omit it from notations and write $\S_{\eps,r}^k$ and $\S_\eps^k$. 
\end{definition}

{

In the paper we will work with both metric tangent space and metric-measure tangent spaces, which are defined as follows. Let  $(X,\sfd,\mm)$ be a m.m.s.,  $p\in \supp(\mm)$ and $r\in(0,1)$; we consider the rescaled and normalized p.m.m.s. $(X,r^{-1}\sfd,\mm^{p}_r,p)$ where the measure $\mm^{p}_r$ is given by
\begin{equation}
\label{eq:normalization}
\mm^{p}_r:=\left(\int_{B_r(p)}1-\frac 1r\sfd(\cdot,p)\,\mm\right)^{-1}\mm.
\end{equation}
Then we define:
\begin{definition}[The collection of m.m. tangent spaces $\Tan(X,\sfd,\mm,p)$]\label{def:tang}
Let  $(X,\sfd,\mm)$ be a m.m.s. and  $p\in \supp(\mm)$. A p.m.m.s.  $(Y,\sfd_Y,\mm_{Y},y)$ is called a (metric-measure)
\emph{tangent { space}} to $(X,\sfd,\mm)$ at $p \in X$ if there exists a sequence of radii $r_i \downarrow 0$ so that
$(X,r_i^{-1}\sfd,\mm^{p}_{r_i},p) \to (Y,\sfd_Y,\mm_{Y},y)$ as 
$i \to \infty$ in the pointed measured Gromov-Hausdorff topology.
We denote the collection of all the tangents of $(X,\sfd,\mm)$ at 
$p \in X$ by $\Tan(X,\sfd,\mm,p)$ or, more shortly when there is no ambiguity, by $\Tan(X,p)$. 
\end{definition}

Analogously, given a metric space $(X,\sfd)$ and a point $p\in X$ we say that a pointed metric space  $(Y,\sfd_Y,y)$ is a (metric) tangent  if there exists a sequence of radii $r_i \downarrow 0$ so that
$(X,r_i^{-1}\sfd,p) \to (Y,\sfd_Y,y)$ as $i \to \infty$ in the pointed Gromov-Hausdorff topology. We denote  with $\Tan(X,\sfd, p)$ (or, more shortly when there is no ambiguity, with $\Tan(X,p)$) the collection of all  metric tangents of $X$ at $p$.

Notice that if $(X,\sfd,\mm)$ satisfies $\RCD(K,N)$ (or more generally a local doubling condition), then Gromov's Compactness Theorem ensures that the set  $\Tan(X,\sfd,\mm,p)$ is non-empty. Notice however that in general (and actually very often in the non-smooth setting) there is more than one element both in  $\Tan(X,\sfd,\mm,p)$ and in $\Tan(X,\sfd,p)$.
\\
{ Note that, by the very definition \eqref{eq:defRNeps} of  $({\mathcal R}_{N})_{\eps}$, if $x \in  ({\mathcal R}_{N})_{\eps}$ then for every $(Y,y)\in {\rm Tan}(X,x)$ it holds 
\begin{equation}\label{eq:tanRN}
\sfd_{GH}(B_{1}^{Y}(y), B^{\R^{N}}_{1}(0^{N}))\leq \eps.
\end{equation}
}
\\We next relate the symmetry of the tangent space with the singular sets $\S_\eps^k$.}
\\It is easy to see that if $(X,\sfd,\mm)$ is $\RCD (K,N)$ then   $\S_\eps^k\subset \hat \S_\eps^k$ where $\hat \S_\eps^k$ is the set of points $p\in X$ such that for \emph{any} tangent space  $Y\in {\rm Tan}(X,\sfd,p)$, the unit ball around the vertex is not  $(k+1,\eps)$-symmetric.

Recall that the singular stratum $\S^k(X)$ is defined as the set of points $p\in X$ such that no element of ${\rm Tan}(X,\sfd,p)$  is $(k+1)$-symmetric. { From the very definitions, it is clear that $\cup_{\eps>0}\hat \S_\eps^k(X) \subset \S^{k}(X)$. Also the reverse inclusion $\S^{k}(X) \subset \cup_{\eps>0}\hat \S_\eps^k(X)$ holds: indeed if $x\in \S^{k}(X)$ then a standard compactness-contradiction argument gives that there exists $\eps>0$ such that $x\in \hat \S_\eps^k(X)$.
\\We claim that it also holds $\S^{k}(X)= \cup_{\eps>0}\S_\eps^k(X)$. Indeed if $x\in \S^{k}(X)$ then, by a standard compactness-contradiction argument, it is readily seen that there exists $\eps>0$ such that $x\in \S_\eps^k(X)$. The converse inclusion trivially follows by $\S_\eps^k \subset  \hat \S_\eps^k \subset \S^{k}$, for every $\eps>0$. We conclude that
\begin{equation}
\S^k(X)=\cup_{\eps>0}\S_\eps^k(X)=\cup_{\eps>0}\hat \S_\eps^k(X).
\end{equation}
}
%
%
We finally set  $\S(X):=\bigcup_{k} \S^k(X)$ to be the singular set. 

\begin{remark}\label{quant-scale}
It is also immediate from the definition that  for any $\lambda\ge 1$ it  holds that $\S_\eps^k(X, \sfd)\subset  \S_\eps^k(X,\lambda \sfd)$ and $\hat \S_\eps^k(X,\sfd)= \hat \S_\eps^k(X,\lambda \sfd)$.
\end{remark}
A key role will be played by the following metric Reifenberg-type result proved by Cheeger and Colding \cite[Theorem A.1.1]{CC97}. 

\begin{theorem}[Cheeger-Colding metric Reifenberg Theorem]\label{thm:CCReif}
Fix $N\in \N, N\geq 1$ and $\alpha\in (0,1)$.
There exists $\bar{\varepsilon}=\bar{\varepsilon}(N, \alpha)>0$, with the following properties. Let $(X,\sfd)$ be a complete  metric space such that for some $\bar{x}\in X$ and $\varepsilon\in (0, \bar{\varepsilon}]$ it holds that
\begin{equation}
x\in (\cR_{N})_{\varepsilon, r}, \quad \text{for all }x\in B_{1}^{X}(\bar{x}) \text{ and } r\in (0, 1-\sfd(\bar{x},x)].
\end{equation}
Then there exists a topological embedding $F:B_{1}^{\R^{N}}(0)\to B_{1}^{X}(\bar{x})$ such that $F(B_{1}^{\R^{N}}(0))\supset B_{\alpha}^{X}(\bar{x})$.
Moreover,  the maps $F, F^{-1}$ are H\"older continuous, with exponent $\alpha$. Further, both $F$ and $F^{-1}$ are  $\Psi(\eps|N)$-GH-approximations between $B_1(\bar x)$ and $B_1(0)$.
\end{theorem}

This theorem has the following generalization~\cite[Theorems A.1.2, A.1.3]{CC97}.

\begin{theorem}\label{reifenberg-stab}
Fix $N\in \N, N\geq 1$ and $\alpha\in (0,1)$.
There exists $\eps_0=\eps_0(N,\alpha)>0$ with the following property.

 If $(X,\sfd)$ is a complete metric space such that $X= (\cR_{N})_{\varepsilon, r}(X)$, for some $r>0$ and $\eps<\eps_0$, then $X$ is homeomorphic to a  {smooth} manifold.
 
 Moreover, if $X_1, X_2$ are two such metric spaces which in addition satisfy $\sfd_{GH}(X_1,X_2)<\eps$ then there exist  $\alpha$-biHolder embeddings $f_1\co X_1\to X_2$ and
 $f_2\co X_2\to X_1$ which are also $\Psi(\eps|N)$ GH-approximations.
 
 In particular, if both $X_1,X_2$ are closed manifolds then $f_1,f_2$ are $\alpha$-bi-H\"older homeomorphisms.
\end{theorem}

Theorem ~\ref{reifenberg-stab} immediately implies the following mild generalization of \cite[Theorem A.1.3]{CC97}

\begin{theorem}\label{reifenberg-stab-pointed}
Fix $N\in \N, N\geq 1$ and $\alpha\in (0,1)$.
There exists $\eps_0=\eps_0(N,\alpha)>0$ with the following property.

Let $(M^N,g,p)$ be a complete connected pointed  Riemannian manifold and let $\{(X_{i}, \sfd_{i}, p_{i})\}_{i\in \N}$ be a sequence of complete pointed  metric spaces, converging to $(M^N,g,p)$ in pointed Gromov Hausdorff sense.  
\\Suppose for any $R>0$ there is $r(R)>0$ such that for all large $i$ it holds that  $B_R(p_i)\subset (\cR_{N})_{\varepsilon_{0}, r(R)}(X_i)$.

Then for any $R>0$, for all large $i$, $B_R(p_i)$ is an $N$-manifold and  there is an $\alpha$-bi-H\"older  embedding $f_{i,R}\co B_{R}(p_i)\to B_{R+\eps_i}(p)$ which is also an $\eps_i$-GH approximation with $\eps_i\to 0$ and $d(p, f_{i,R}(p_i))<\eps_i$ and  such that

$f_{i,R}(B_R(p_i))\supset B_{R-10 \eps_i}(p)$ and $\sfd_{(M,g)}(f_{i,R}(p_i),p)\le \eps_i$.

\end{theorem}
\begin{remark}\label{rem-onto}
 \cite[Theorem A.1.3]{CC97} directly implies the above theorem in case of compact $M$. However, the proof of  \cite[Theorem A.1.3]{CC97}  actually gives the above pointed version as well except possibly for the inclusion $f_{i,R}(B_R(p_i))\supset B_{R-10\eps_i}(p)$.
But that inclusion  easily follows from the other conclusions of the theorem (cf. \cite[Lemma 4.8]{Kap07}).
Indeed, the intersection $f_{i,R}(\bar B_{R-\eps_i}(p_i))\cap \bar B_{R-10 \eps_i}(p)$ is clearly nonempty and closed in  $\bar B_{R-10 \eps_i}(p)$ and by the invariance of domain theorem it's also open in  $\bar B_{R-10 \eps_i}(p)$. Hence, it's equal to $\bar B_{R-10 \eps_i}(p)$.
\end{remark}
\begin{corollary}\label{reifenberg-stab-compact}
If under the assumptions of Theorem ~\ref{reifenberg-stab-pointed}    the manifold $M$ is compact then   {$X_{i}$} is bi-H\"older homeomorphic to $M$ for all large $i$.
\end{corollary}
The next two results were proved in \cite{GDPNonColl} extending to the $\RCD$ setting  celebrated results by Colding \cite{Col97,coldingshape} (see also \cite{CC97}).
\begin{theorem}[GH-Continuity of $\cH^{N}$]\label{thm:ContHN}\cite[Theorem 1.3]{GDPNonColl} 
Fix some $K\in \R$, $N\in[1,\infty)$ and $R\geq0$. Let ${\mathbb B}_{K,N,R}$ be the collection of all (equivalence classes up to isometry of) closed balls of radius $R$
in $\RCD(K,N)$ spaces  equipped with the Gromov-Hausdorff distance.
\\Then the map ${\mathbb B}_{K,N,R}\ni Z \mapsto \cH^{N}(Z)$  is real valued and continuous.
\end{theorem}

\begin{theorem}[Volume Rigidity]\label{thm:VolRig}\cite[Theorem 1.6]{GDPNonColl}
For every $\varepsilon>0$ and $N\in \N$, $N\geq 1$, there exists $\delta=\delta(\varepsilon,N)>0$ such that the following holds. 
Let $(X,\sfd, \cH^{N})$ be a non-collapsed $\RCD(-\delta,N)$  space and assume there exists $\bar{x}\in X$ satisfying $\cH^{N}(B_{1}^{X}(\bar{x})) \geq (1-\delta) \cH^{N}(B_{1}^{\R^{N}}(0) )$.Then 
$$\sfd_{GH}(\bar{B}_{1/2}^{X}(\bar{x}), \bar{B}_{1/2}^{\R^{N}}(0))\leq \varepsilon.$$
\end{theorem}

{
Combining the GH-continuity of $\cH^{N}$ (Theorem \ref{thm:ContHN}) with the  volume rigidity (Theorem \ref{thm:VolRig}) gives the next characterization of  $(\cR_{N})_{\eps}$ in terms of the density $\vartheta_{N}$ defined in \eqref{eq:defthetaN}.
\begin{corollary}\label{cor:RNepsTheta}
Let $(X,\sfd, \cH^{N})$ be a non-collapsed $\RCD(K,N)$  space for some $K\in \R, N\in \N$ and let $x\in X$. The following holds: 
\begin{itemize}
\item If $\vartheta_{N}(x)\geq 1-\eps$ then $x\in (\cR_{N})_{\Psi(\eps|K,N)}$ .
\item Conversely, if   $x\in (\cR_{N})_{\eps}$ then  $\vartheta_{N}(x)\geq 1-\Psi(\eps|K,N)$.
\end{itemize}
\end{corollary}
}

Let $(\cWR_N)_\eps(X)$ denote the set of points in $X$ such that \emph{some} tangent space  $T_xX$ satisfies  
\[\sfd_{GH}(B_1^{T_xX}(o), B_{1}^{\R^{N}}(0))\le \eps.
\]

Combining the above two theorems and Bishop-Gromov volume comparison (see for instance the proof of Theorem \ref{thm:topologyHN}) we get: 

\begin{corollary}\label{cor-reg}
 For any $N\in\N, K\in \R$ there exists $\eps(\delta,K,N)=\Psi(\delta|K,N)$ such that 
if  $(X,\sfd,\cH^{N})$ is a non-collapsed $\RCD(K,N)$ space then  $X\backslash \S^{N-1}_\delta\subset (\cR_{N})_{\eps(\delta|K,N)}$  and $(\cWR_N)_\eps\subset (\mathcal R_N)_{\Psi(\eps|K,N)}$.
\end{corollary}

\section{Topological regularity}

The next theorem extends to the $\RCD$ setting a celebrated  result by Cheeger-Colding \cite[Theorem A.1.8]{CC97}
{ 
\begin{theorem}\label{thm:topologyHN}
Fix $K\in \R$ and $N\in \N, N\geq 1$.  For every $\alpha\in (0,1)$, there exist $\bar\varepsilon=\bar\varepsilon(K,N,\alpha)>0$ { such that for any $0<\eps\le \bar\varepsilon$ we can find  $\bar{r}=\bar{r}(K,N,\alpha,\varepsilon)$ satisfying the next assertion}.
\\Let $(X,\sfd, \cH^{N})$ be an $\RCD(K,N)$ space and let $\bar{x}\in (\cR_{N})_{\varepsilon,r}$ be a  $(N,\varepsilon,r)$-regular point, for some $r\in (0, \bar{r})$.
Then there exists a topological embedding $F:B_{\alpha r}^{\R^{N}}(0)\to B_{\alpha r}^{X}(\bar{x})$ such that $F(B_{\alpha r}^{\R^{N}}(0))\supset B_{\alpha r}^{X}(\bar{x})$.
Moreover,  the maps $F, F^{-1}$ are H\"older continuous, with exponent $\alpha$.  Further, both $F$ and $F^{-1}$ are  $\Psi(\eps|N)$-GH-approximations.
\end{theorem}
}

\begin{proof}

First of all we fix $K\in \R$, $N\in \N, N\geq 1$, $\alpha\in (0,1)$ and an  $\RCD(K,N)$ space  $(X,\sfd, \cH^{N})$.  
\\Notice that there exists $\bar{r} =\bar{r}(K,N,\delta)$ such that  the rescaled space $(X,\sfd/\bar{r}, \cH^{N})$ is  $\RCD(-\delta, N)$; in the latter space, $\cH^{N}$ is the Hausdorff measure corresponding to the rescaled distance $\sfd/{\bar{r}}$. In order to keep the notation short, let us denote $X/{\bar{r}}:=(X,\sfd/{\bar{r}},  \cH^{N}_{\sfd/{\bar{r}}})$. 
By definition of $(N,\varepsilon, {\bar{r}})$-regular point, it holds 
\begin{equation}\label{eq:dGHX/r}
\sfd_{GH}(B_{1}^{X/{\bar{r}}}(\bar{x}), B_{1}^{\R^{N}}(0))\leq \varepsilon.
\end{equation}
The GH-continuity of $\cH^{N}$ (see Theorem \ref{thm:ContHN}) combined with \eqref{eq:dGHX/r} gives that 
\begin{equation}\label{eq:HNMax}
\cH^{N}(B_{1}^{X/\bar{r}}(\bar{x})) \geq (1-\Psi(\varepsilon|N))\, \cH^{N}(B_{1}^{\R^{N}}(0)).
\end{equation}
We now claim that \eqref{eq:HNMax} combined with Bishop-Gromov monotonicity of the volume  implies that any point $x\in B_{\eta}(\bar{x})$ has almost maximal volume at every (sufficiently small) scale, i.e.:
\begin{equation}\label{eq:AlmMaxVolBrho}
\frac{\cH^{N}(B^{X/r}_{\rho}(x))}{\omega_{N} \rho^{N}}\geq 1-\Psi(\varepsilon,\delta, \eta|N), \quad \text{ for all } x\in B_{\eta}(\bar{x}), \,  \rho\in (0, 1). 
\end{equation}
Indeed using that $B^{X/r}_{1}(\bar{x})\subset B^{X/r}_{1+\eta}(x)$ for every $x\in B^{X/r}_{\eta}(\bar x)$ and recalling Bishop-Gromov inequality \eqref{eq:BSballs},  we obtain that for every $\rho\in (0, 1+\eta)$ it holds
\begin{align}
\frac{\cH^{N}(B^{X/r}_{\rho}(x))}{v_{-\delta,N}(\rho)}& \geq  \frac{\cH^{N}(B^{X/r}_{1+\eta}(x))}{v_{-\delta,N}(1+\eta)} \geq  \frac{\cH^{N}(B^{X/r}_{1}(\bar x))}{v_{-\delta,N}(1+\eta)} \nonumber \\
& \overset{\eqref{eq:HNMax}}{\geq}   (1-\Psi(\varepsilon|N))\, \frac{\cH^{N}(B_{1}^{\R^{N}}(0))}{v_{-\delta,N}(1+\eta)} \geq 1-\Psi(\varepsilon,\delta, \eta|N). \label{HNvdelta}
\end{align}
The claim \eqref{eq:AlmMaxVolBrho} follows from  \eqref{HNvdelta} combined with the estimate $v_{-\delta,N}(\rho)\geq (1-\Psi(\delta|N)) \omega_{N}  \rho^{N} $ for every $\rho\in (0, 2)$.
\\In virtue of Theorem \ref{thm:VolRig},  the claim \eqref{eq:AlmMaxVolBrho} implies that
\begin{equation}\label{eq:dGHBrho}
\sfd_{GH}(\bar{B}_{\rho}^{X/r}(x), \bar{B}_{\rho}^{\R^{N}}(0))\leq \Psi(\varepsilon,\delta, \eta|N) \rho,  \quad \text{ for all } x\in B^{X/r}_{\eta}(\bar{x}), \,  \rho\in (0, 1/2).
\end{equation}
In other terms,  coming back to the orginal scale of $(X,\sfd)$, we have just proved that  $B^{X}_{\eta r}(\bar{x})\subset (\cR_{N})_{\Psi(\varepsilon,\delta, \eta|N),r/2}.$
\\ The conclusion follows now from Theorem \ref{thm:CCReif}.
\end{proof}

{Combining \eqref{eq:RNfullMeas} with Theorem \ref{thm:topologyHN} we obtain:}
\begin{corollary}
Let $(X,\sfd, \cH^{N})$ be an $\RCD(K,N)$ space for some $K\in \R$,  $N\in \N, N\geq 1$. Fix $\alpha\in (0,1)$.
\\Then there exists an open subset $U\subset X$ such that
\begin{itemize}
 \item $\cR_{N}\subset U$. In particular $U$ is dense in $X$ and of full measure, i.e. $\mm(X\setminus U)=0$.
\item $U$ is a $C^{\alpha}$-manifold, i.e. it is a topological manifold  with $C^{\alpha}$ charts.
 \end{itemize}
\end{corollary}

The following theorem is a stronger version of Theorem~\ref{thm:ConvergenceMan} which includes possibly noncompact limits.
\begin{theorem}[Topological stability]\label{top-stability-RCD}
Let $0<\alpha<1$.
Suppose $(X_i,p_i)\to (M^N,p)$ is a pmGH converging pointed sequence of non-collapsed $\RCD(K,N)$ spaces where $M$ is a smooth Riemannian manifold. Then
for any fixed $R>0$ for all large $i$  it holds that the balls $B_R(p_i)$ are  topological manifolds and there exist $\alpha$-bi-H\"older embeddings $f_{i}:B_R(p_i)\to  {B_{R+\eps_i}(p)}$ which are also $\eps_i$-GH approximations with $\eps_i\to 0$ and such that
$f_i(B_R(p_i))\supset B_{R-10\eps_i}(p)$ and $\sfd_{(M,g)}(f_{i,R}(p_i),p)\le \eps_i$.

In particular, if $M$ is compact then $X_i$ is $\alpha$-bi-H\"older homeomorphic to $M$ for all large $i$.
\end{theorem}
\begin{proof}
Let $\eps>0$ be an arbitrary  positive real number. The same argument as in the proof of Theorem ~\ref{thm:topologyHN} shows that for any fixed $R>0$ there is $r>0$ such that for all large $i$ all points in $B_R(p_i)$ lie in $ (\cR_{N})_{\varepsilon, r}(X_i)$. Now the result follows by Theorem ~\ref{reifenberg-stab-pointed} if $\eps>0$ is chosen small enough.
\end{proof}
\begin{remark}
The same proof shows that Theorem~\ref{top-stability-RCD} remains true if $M$ is a non-collapsed $\RCD(K,N)$ space with all points lying in $(\cR_{N})_{\eps_1}$  for some sufficiently small $\eps_1=\eps_1(N)$.
\end{remark}

Using Bishop-Gromov inequality and arguing similarly to the proof of Theorem \ref{thm:topologyHN}, we obtain the next two rigidity and almost rigidity results which are the $\RCD$ counterparts of \cite[Theorem A.1.10, A.1.11]{CC97} established by Cheeger-Colding for smooth Riemannian manifolds.
\\In order to state the next result, let us recall the notion of tangent space at infinity for an $\RCD(0,N)$ space $(X,\sfd, \cH^{N})$. First of all note that, if $(X,\sfd, \cH^{N}_{\sfd})$ is $\RCD(0,N)$, then for every $r>0$ the rescaled space  $(X,r \,\sfd, \cH^{N}_{r \, \sfd})$ is still $\RCD(0,N)$; Gromov's Compactness Theorem thus implies that, for any sequence $r_{i}\downarrow 0$, the sequence $(X,r_{i}\, \sfd, \cH^{N}_{r_{i}\, \sfd})$ admits a subsequence which is pmGH-converging to some $\RCD(0,N)$ space, called \emph{a tangent space at infinity}.  In  general the tangent space at infinity may not be unique and need not be a metric cone (see for instance \cite[Example 8.95]{CC97}).

\begin{theorem}\label{Thm:ConeInf}
Fix some  $N\in \N$. Let $(X,\sfd, \cH^{N})$ be an  $\RCD(0,N)$ space.
\\If $X$ admits a tangent space at infinity isometric  to $\R^{N}$ then $(X,\sfd, \cH^{N})$ is isomorphic to $\R^{N}$ as a m.m.s..
\\Moreover  the  following almost rigidity holds.  For every $N\in \N, N\geq 1$ there exists $\varepsilon(N)>0$ with the following property. 
\begin{itemize}
\item If there exists $p\in X, R_{0}\geq 0$ such that $\sfd_{GH}(B_{r}(p),B_{r}(0^{N}))\leq \varepsilon r$ for all $r\geq R_{0}$,  then $X$ is homeomorphic to $\R^{N}$ and $\cH^{N}(B_{r}(p))\geq (1-\Psi(\varepsilon|N))  \omega_{N} r^{N}$ for all $r\geq R_{0}$;
\item If there exists $x\in X, R_{0}\geq 0$ such that $\cH^{N}(B_{r}(p))\geq (1-\varepsilon)  \omega_{N} r^{N}$ then $X$ is homeomorphic to $\R^{N}$ and $\sfd_{GH}(B_{r}(p),B_{r}(0^{N}))\leq \Psi(\varepsilon|N)  r$ for all $r\geq R_{0}$.
\end{itemize}
\end{theorem}

\begin{proof}
The proof of Theorem ~\ref{Thm:ConeInf}  follows from the proof of Theorem~\ref{top-stability-RCD} in verbatim the same way as in the proof of \cite[Theorem  A.1.11]{CC97} (which is a smooth analog  of Theorem ~\ref{Thm:ConeInf} )  from \cite[Remark A.1.47]{CC97}.

Instead of using \cite[Remark A.1.47]{CC97} one can also argue as follows. It follows from Volume Rigidity (Theorem~\ref{thm:VolRig}) and GH-Continuity of $\cH^{N}$ (Theorem~\ref{thm:ContHN}) that the assumptions of the bullet  points are equivalent. We will therefore assume that both hold.
Volume rigidity and Theorem ~\ref{thm:topologyHN} easily imply that $X$ is a topological $N$-manifold.

The main part is to prove that $X$ is homeomorphic to $\R^N$.

  We have that for any $\eps>0$  there is a large enough $R_0=2^k$ so that for all $r\ge R_0$ it holds that $\sfd_{GH}(B_{r}(x),B_{r}(0^{N}))\leq \eps r$. When $\eps$ is small enough, by Theorem~\ref{top-stability-RCD}, this implies that there exists a bi-H\"older embedding $F_0\co B_{R_0}(0)\to  B_{(1+\Psi(\eps|N))R_0}(p)$   whose image contains  $B_{R_0}(p)$ and which is a $\Psi(\eps|N)R_0$-GH-approximation from  $B_{R_0}(0)\subset \R^N$ to $ B_{R_0}(p)$. 

For each $i>0$ we can get similarly constructed maps $F_i\co B_{1.1R_i}(0)\to  B_{(1.1+\Psi(\eps|N))R_i}(p)$ where $R_i=2^{k+i}$. For any $i>0$ let $C_i$ be the annulus in $\R^N$ equal to $\{0.4R_i<|x|<1.1R_i\}$ and let $A_i$ be the annulus in $X$ given by $\{0.4 R_i<\sfd(p, \cdot)< 1.1R_i\}$. The maps $F_i$ give bi-H\"older embeddings $C_i\to \{(0.4-\Psi(\eps|N))R_i<\sfd(p,\cdot)< (1.1+\Psi(\eps|N))R_i\}$ which are also $\Psi(\eps|N)R_i$-GH-approximations from $C_i$ to $A_i$. Note that the maps $G_i=F_i^{-1} \co A_i\to \R^N$ and $G_{i+1}=F_{i+1}^{-1} \co A_{i+1}\to \R^N$ need not a-priori be uniformly close on the overlaps but they can be made close by post-composing with orthogonal maps in $\R^N$.  Indeed the composition $f_i=G_{i+1}\circ F_i$ gives a $\Psi(\eps|N)R_i$-GH-approximation from  $B_{1.1R_i}(0)\subset \R^n$ to itself. Rescaling the domain and the target by $1/R_i$ this means that $f_i\co B_{1.1}(0)=\frac{1}{R_i} B_{1.1R_i}(0)\to B_{1.1}(0)=\frac{1}{R_i} B_{1.1R_i}(0)$ is an $\Psi(\eps|N)$-GH-approximation. By the limit argument it's $\Psi(\eps|N)$-close to a self  isometry of $B_{1.1}(0)$, i.e to a linear map given by some orthogonal matrix $B_i$. Going back to the unrescaled spaces this means that $f_i\co B_{1.1R_i}(0)\to B_{1.1R_i}(0)$ is  $\Psi(\eps|N)R_i$-close to $B_i$.
Thus, after post composing  $G_{i+1}$ with an orthogonal map we can assume that $G_i$ and $G_{i+1}$ are $\Psi(\eps|N)R_i$-close on
$A_i\cap A_{i+1}$ for all $i\ge 0$.

 Note that $G_i$ is an embedding which is also an $R_i\Psi(\eps)$-GH-approximation from  $A_i$ to the annulus $C_i$ in $\R^N$.

Using  Siebenmann's deformation of homeomorphisms theory \cite{Sieb},  this implies  (\cite[Theorem 4.11]{Kap07}) that if $\eps>0$ is small enough then  for each $i\ge 0$ it's possible to modify  $G_i$ and $G_{i+1}$ on a small neighborhood of $A_i\cap A_{i+1}$ such that they become equal there and the ``glued" map $ G\co X\to \R^N$  is still a topological embedding.
 By the same argument as in Remark~\ref{rem-onto} the map $G$ is onto i.e. it's a homeomorphism.

\end{proof}

\begin{theorem}[Sphere Theorem]\label{thm:Sphere}
For every $N\in \N, N\geq 1$ there exists $\varepsilon(N)>0$ with the following property. 
\\Let $(X,\sfd, \cH^{N})$ be an $\RCD(N-1,N)$ space for some  $N\in \N, N\geq 1$.
\begin{itemize}
\item If $\sfd_{GH}(X,\bS^{N})\leq \varepsilon$ then $X$ is homeomorphic to $\bS^{N}$ and $\cH^{N}(X)\geq (1-\Psi(\varepsilon|N))  \cH^{N}(\bS^{N})$;
\item If $\cH^{N}(X)\geq (1-\varepsilon) \cH^{N}(\bS^{N})$ then  $X$ is homeomorphic to $\bS^{N}$ and $\sfd_{GH}(X,\bS^{N})\leq \Psi(\varepsilon|N)$.
\end{itemize}
\end{theorem}

{

\begin{proof} 
Both the statements can be proved by a compactness-contradiction argument.

We first discuss the first statement: if $(X,\sfd_{i}, \cH^{N}), i\in \N,$ is a sequence of $\RCD(N-1,N)$ spaces such that $(X,\sfd_{i})$ is GH-converging to $\bS^{N}$, then  the GH-continuity of $\cH^{N}$ (Theorem \ref{thm:ContHN}) implies that  $\cH^{N}(X_{i})\to  \cH^{N}(\bS^{N})$ and thus  $(X,\sfd_{i}, \cH^{N})$ converges to $\bS^{N}$ in mGH sense.  We conclude that $(X,\sfd_{i})$ is homeomorphic to $\bS^{N}$ by the topological stability Theorem \ref{top-stability-RCD}.

Regarding the second statement:  let $(X,\sfd_{i}, \cH^{N}), i\in \N,$ be a sequence of $\RCD(N-1,N)$ spaces such that $\cH^{N}(X_{i})\to  \cH^{N}(\bS^{N})$. By Gromov's Compactness Theorem there exists an $\RCD(N-1,N)$ space $(Y,\sfd_{Y}, \mm_{Y})$ such that, up to subsequences,  $(X,\sfd_{i}, \cH^{N})\to (Y,\sfd_{Y}, \mm_{Y})$ in mGH sense; moreover, by the stability of non-collapsed $\RCD(K,N)$ spaces (see \cite[Theorem 1.2]{GDPNonColl}) it follows that $\mm_{Y}=\cH^{N}$. In particular $\cH^{N}(Y)=\lim_{i\to \infty}  \cH^{N}(X_{i})=  \cH^{N}(\bS^{N})$. By Bishop-Gromov inequality \eqref{eq:BSballs} it follows that ${\rm rad}(Y)=\pi$, where
$$
{\rm rad}(Y):= \inf_{x\in Y} \sup_{y\in Y} \sfd_{Y}(x,y)=\inf\{r>0 \, :\,  \exists\, x\in Y \text{ s.t. } Y\subset B_{r}(x)\} 
$$
is the radius of $(Y,\sfd_{Y})$. An iteration of the Maximal Diameter Theorem \cite{Ket} gives that  $(Y,\sfd_{Y}, \mm_{Y})$ is isomorphic as metric measure space to $\bS^{N}$. Thus $(X,\sfd_{i})\to \bS^{N}$ in GH-sense and $(X,\sfd_{i})$ is homeomorphic to $\bS^{N}$ by the topological stability Theorem \ref{top-stability-RCD}.
\end{proof}
}

\section{Boundary of a non-collapsed $\RCD$ space}  
In this section we define the boundary of a non-collapsed  $\RCD(K,N)$ space and study its properties. 

At the core of this definition is the following lemma.

\begin{lemma}\label{lem:TanncRCD}
Let $(X,\sfd,\cH^{N})$ be a non-collapsed $\RCD(K,N)$ space. Then for every $x\in X$,  every  $Y\in {\rm Tan}(X,\sfd, \cH^{N}, x)$ is a metric-measure cone over a non-collapsed $\RCD(N-2,N-1)$ space $Z$, i.e. $Y=C(Z)$.
\end{lemma}

\begin{proof}
First of all recall that tangent  { spaces} to non-collapsed $\RCD(K,N)$ spaces are metric measure cones and are non-collapsed $\RCD(0,N)$ spaces (see \cite[Step 2, Page 645]{GDPNonColl}), i.e.  for every $x\in X$,  every  $Y\in {\rm Tan}(X,\sfd, x)$ is a  metric-measure cone and a non-collapsed $\RCD(0,N)$ space. Thus there exists a m.m.s $(Z,\sfd_{Z}, \mm_{Z})$ such that $Y=C(Z)$.  By \cite[Theorem 1.2]{Ket} it follows that $Z$ satisfies $\RCD(N-2,N-1)$. Using on the one hand that $Y$ is non-collapsed $\RCD(0,N)$ and on the other hand that the metric-measure structure on $Y$ is given by the $(0,N)$-cone structure $(Y, \sfd_{Y}, \mm_{Y}=\cH^{N}_{\sfd_{Y}})=(C(Z), \sfd_{C}, \mm_{C,N})$, from the definitions \eqref{eq:defConeDist}-\eqref{eq:ConeMeas},  it is easy to check that $\vartheta_{N-1}[(Z,\sfd_{Z},\mm_{Z})]=1$ $\mm_{Z}$-a.e..
 Indeed, the metric-measure cone structure  \eqref{eq:defConeDist}-\eqref{eq:ConeMeas}  implies that for any $z\in Z, t>0$ it holds that $\vartheta_N(t,z)=\vartheta_{N-1}(z)$ where the first density is with respect to $\mm_{Y}$ and the second is with respect to $\mm_{Z}$. Since $\vartheta_{N}=1$ $\mm_{Y}$-a.e. on $Y$ this immediately yields that $\vartheta_{N-1}=1$ $\mm_{Z}$-a.e. on $Z$.
By \eqref{noncoll-density-1} we conclude that  $Z$ is a non-collapsed $\RCD(N-2,N-1)$ space.
\end{proof}

\begin{definition}[Boundary of a non-collapsed $\RCD(K,N)$ space]\label{def:bRCD}
Given $(X,\sfd,\cH^{N})$, a non-collapsed $\RCD(K,N)$ space for some $K\in \R, N\in \N$, we define the {$\RCD$-}boundary of $X$ as
\begin{equation}\label{eq:defBoundary}
\bRCD X:=\{x\in X \,:\, \exists Y\in {\rm Tan}(X,\sfd,x), \, Y=C(Z), \, \bRCD Z\neq \emptyset\},
\end{equation}
and the reduced boundary of $X$ as
\begin{equation}\label{eq:defBoundary}
\bRCD^* X:= \S^{N-1}\setminus \S^{N-2}.
\end{equation}
\end{definition}

Note that, thanks to Lemma \ref{lem:TanncRCD} the definition of $\bRCD X$ is inductive on the dimension $N$ of the non-collapsed  $\RCD(K,N)$ space.
The base of this inductive definition lies on the classification of $\RCD(0,1)$ spaces $(X,\sfd, \cH^{1})$ proved in \cite{KL}: such an $(X,\sfd,\cH^{1} )$ is  isomorphic (up to rescaling) to either the singleton $\{x\}$, the unit circle ${\mathbb S}^{1}$, the real line $\R$, the half line $[0,\infty)\subset \R$, or the segment $[0,1]\subset \R$.  Of course, in the first three cases we say that (by definition) $\bRCD X=\emptyset$, in the last two cases we set (again by definition)  $\bRCD X=\{0\}$,  $\bRCD X=\{0,1\}$ respectively.
\\

{

In the next lemma we give a characterization of the $\RCD$-boundary in terms of iterated tangent spaces.  Let us first clarify the meaning of iterated tangent space. Given $x\in X$ we write $T_{x}X$ to denote an arbitrary element in $\Tan(X,\sfd,x)$ (no claim of uniqueness is made here). An $N$-iterated tangent space at $x\in X$ is a metric space   $T_{\xi_N}T_{\xi_{N-1}}\ldots T_{x} X$, where $\xi_1:=x, \xi_2\in T_{\xi_1}X,$ etc. .

\begin{lemma}\label{lem:bRCDXIteratedTang}
Let $(X,\sfd,\cH^{N})$ be a non-collapsed $\RCD(K,N)$ space for some $K\in \R, N\in \N$. Then
\begin{equation}\label{eq:bRCDTangents}
\bRCD X=\{ x\in X \,:\, \exists \text{ an $N$-iterated tangent space  at $x$, i.e. $T_{\xi_N}T_{\xi_{N-1}}\ldots T_{x} X$, isometric to  $\R^N_+$}\},
\end{equation}
where  $\R^N_+:=\{(x_1,\ldots,x_N)| x_1\ge 0\}$ is the $N$-dimensional Euclidean half-space.
\end{lemma}

\begin{proof}
``$\supset$'': follows by the fact that  $\R^N_+$ has non-empty $\RCD$-boundary according to Definition \ref{def:bRCD} since it is a cone over the half sphere $\bS^{N-1}_{+}:=\bS^{N-1} \cap \R^{N}_{+}$, and we can induct on the dimension up to $\R_{+}$.

``$\subset$''. We argue by induction on $N\in \N$. 
\\ Case $N=1$: if $X$ is non-collapsed $\RCD(K,1)$ with non-empty boundary then it is isometric to either a segment or a half line; in both cases, if  $x\in \bRCD X$ then $\Tan(X,\sfd,x)=\{\R_{+}\}$. 
\\ Case $N>1$: if $X$ is non-collapsed $\RCD(K,N)$ and $x\in \bRCD X$ then (by Lemma \ref{lem:TanncRCD} and  Definition \ref{def:bRCD}) there exists $T_{x}X \in \Tan(X,\sfd,x)$ with $T_{x}X=C(X_{1})$, $X_{1}$ non-collapsed $\RCD(N-2,N-1)$ space with $\bRCD X_{1} \neq \emptyset$. Let $x_{1}\in \bRCD X_{1}$ and note that, by inductive assumption, there exists an $N-1$-iterated tangent space  to $X_{1}$ at $x_{1}$ isometric to $\R^{N-1}_{+}$. Choosing $\xi_{2}:=(1,x_{1})\in C(X_{1})=T_{x}X$ we get that $T_{\xi_{2}}T_{x}X=\R\times T_{x_{1}}X_{1}$. Thus $X$ has an $N$-iterated tangent space isometric to $\R^{N-1}_{+}\times \R=\R^{N}_{+}$.
\end{proof}
}

The definition of the boundary of a non-collapsed $\RCD(K,N)$ space is similar to the definition of the boundary of an Alexandrov space. However  unlike in the Alexandrov case where tangent  spaces are known to be unique we don't know at the moment if it's possible to have points where some tangent  spaces have boundary and others don't.
\begin{question}\label{q:boundTan}
Is it true that if \emph{some} tangent space at  some point $p$ in a non-collapsed $\RCD(K,N)$ space $X$ has boundary then \emph{every} tangent space at $p$ has boundary?
\end{question}

\begin{remark}
The answer to Question \ref{q:boundTan} is ``Yes" if $N\le 3$. Indeed, if $N\le 2$ then $X$ is Alexandrov by ~\cite{LS} and hence tangent  spaces are unique. If $N=3$ then any  $Y\in {\rm Tan}(X,\sfd,p)$ has the form  $Y=C(Z)$ where $Z$ is a 2-dimensional Alexandrov space with $curv\ge 1$. For a non-collapsing sequence of Alexandrov spaces it is known that if the elements of the sequence have (resp. don't have) boundary then the same holds for the limit. Therefore the subset of $ {\rm Tan}(X,\sfd,p)$ consisting of tangent  spaces that have boundary is closed and the same is true for its complement.

Since the space of tangent  spaces $ {\rm Tan}(X,\sfd,p)$ is connected (see e.g.  \cite[Lemma 2.1]{LS} or \cite[Lemma 3.2]{Donaldson-Sun}), the claim follows.
\end{remark}

\begin{lemma}\label{lem:TanSpXNC}
Let  $(X,\sfd,\cH^{N})$ be a non-collapsed $\RCD(K,N)$ space.
\\ Then for every $x\in \S^{N-1}\setminus \S^{N-2}$,  { there exists a} tangent space at $x$ isomorphic to the half space $\R^N_+:=\{(x_{1}, \dots, x_{N})\,:\, x_{1} \geq 0\}$.
\\In particular, $\bRCD^* X\subset \bRCD X$  { and  the Hausdorff dimension of  $\bRCD X \setminus \bRCD^* X$ is at most $N-2$.}
\end{lemma}

\begin{proof}
Recall that every tangent space to an $\RCD(K,N)$ space is a non-collapsed  $\RCD(0,N)$ metric cone \cite[Proposition 2.8]{GDPNonColl}. 
Note that
\begin{equation}\label{eq:SN-2N-1}
\S^{N-1}(X)\setminus \S^{N-2}(X)=\{x\in X\,:\, \not \exists T_{x}X \text{ splitting } \R^{N} \text{ but } \exists T_{x}X \text{ splitting } \R^{N-1}\}.
\end{equation}
Now, if $Y$ is an $\RCD(0,N)$ space splitting $\R^{N-1}$ then $Y=Z\times \R^{N-1}$ with $Z$ an $\RCD(0,1)$ space. By the classification of 1-dimensional $\RCD$ spaces \cite{KL}, $Z$ can be isometric to either a singleton $\{0\}$, a half-line $\R^{+}$, a line $\R$, a closed interval  $[a,b]\subset \R$, or a circle ${\mathbb S}^{1}$. 
\\The case $Z=\{0\}$ is excluded since it would imply $Y=\R^{N-1}$ which has infinite $N$-density  $\vartheta_{N}$ and hence is not a non-collapsed $\RCD(0,N)$. The case $Z=\R$ is excluded since it would imply that $Y=\R^{N}$. The cases $Z={\mathbb S}^{1}$ and  $Z=[a,b]$ are excluded since they would imply that 
$Y$ is not a metric cone. 
\\ Thus the only possibility is that $Z$ is a half line $\R^{+}$ and hence $Y$ is isomorphic to the half space $\R^N_+$.
The claim follows from the combination of this last observation with \eqref{eq:SN-2N-1}.
\\{  It is easy to see that  $(\bRCD X \setminus \bRCD^* X)\subset \cS^{N-2}$ (see for instance the beginning of the proof of Theorem \ref{thm-structure}).   By \cite[Theorem  1.8]{GDPNonColl}, for all $k$ it holds that
\begin{equation}\label{H-dim-Sk}
\dim_{\mathcal H}\cS^{k}\le k.
\end{equation}
We conclude that the Hausdorff dimension of  $\bRCD X \setminus \bRCD^* X$ is at most $N-2$.}
\end{proof}

The above notion of boundary  is compatible with the one of topological manifold with boundary, see  Corollary \autoref{manifold-no-boundary}.
\\It is clear that $\bRCD^*X$ need not be closed. For instance if $X=[0,1]\times[0,1]\subset \R^{2}$ then it is easy to see that $\bRCD^*X=(0,1)\times\{0,1\}\cup \{0,1\}\times(0,1)$ while $\bRCD X=[0,1]\times\{0,1\}\cup \{0,1\}\times[0,1]$. It is not clear if in general $\bRCD X\subset X$ is a closed subset.

\begin{question}
 Let $(X,\sfd,\cH^{N})$ be a non-collapsed $\RCD(K,N)$ space.  Is $\bRCD X$ necessarily a closed subset of $X$?
 \end{question}

 A closely related question is the following:
\begin{question}\label{bry-reduced-bry}
Is it true that if $X$ is a non-collapsed $\RCD(K,N)$ space and $\bRCD X\ne\emptyset$ then $\bRCD^* X\ne\emptyset$ also?
\end{question}

Another closely related question is

\begin{question}\label{bry-closure-reduced-bry}
Let $X$ be a non-collapsed $\RCD(K,N)$ space. Is it true that $\bRCD X$ is equal to the closure of $\bRCD^* X$?
\end{question}
A positive answer to this question would mean that our definition of the boundary is equivalent to the one suggested by De Philippis and Gigli in ~\cite{GDPNonColl}.

The following conjecture is inspired by the theory of finite perimeter sets, and in particular by De Giorgi's Theorem.

\begin{conjecture}
 Let $(X,\sfd,\cH^{N})$ be a non-collapsed $\RCD(K,N)$ space.   Then  $\bRCD X$ (or, equivalently, $\bRCD^{*} X$ in view of Lemma \ref{lem:TanSpXNC}) is $\cH^{N-1}$ rectifiable.
 \end{conjecture}

The next result says that a non-collapsed $\RCD(K,N)$ space  $X$  is the disjoint union of a manifold part of dimension $N$ which is open in $X$, a boundary part of Hausdorff dimension at most $N-1$, and a singular set of Hausdorff dimension at most $N-2$.

\begin{theorem}\label{thm-structure}
Let $(X,\sfd,\cH^{N})$ be a non-collapsed $\RCD(K,N)$ space for some $K\in \R, N\in \N$.
\\Then $\bRCD X\subset \S^{N-1}$, in particular the Hausdorff dimension of $\bRCD X$ is at most $N-1$.
\\Moreover  there exists $\eps_{0}=\eps_{0}(K,N)$ such that for  $\eps\in (0,\eps_{0})$ 
the following properties hold

\begin{enumerate}
\item \label{regular-part} $\mathcal R_N\subset \overset{\circ}{({\mathcal R}_{N})_{\eps}}\subset X$ and $\overset{\circ}{({\mathcal R}_{N})_{\eps}}\subset X$ is $\alpha(\eps)$-bi-H\"older homeomorphic to a smooth  manifold, where $\alpha(\eps)\to 1$ as $\eps\to0$.
\item\label{regular-no-bry} ${ ({\mathcal R}_{N})_{\eps}}\cap \bRCD X=\emptyset$.
\item \label{reg-intr-sing} ${ ({\mathcal R}_{N})_{\eps}} \cap \S\subset \S^{N-2}$, in particular it has Hausdorff dimension at most $N-2$.
\item\label{reg-connected} If $\cH^{N-1}(\S)=0$ (equivalently, if  $\cH^{N-1}(\bRCD^* X)=0$), then  $\overset{\circ}{({\mathcal R}_{N})_{\eps}}$ is path connected. Moreover, the induced inner metric on $\overset{\circ}{({\mathcal R}_{N})_{\eps}}$ coincides with the restriction of the ambient metric $\sfd$.
\end{enumerate}
It follows that 
$$
X= \overset{\circ}{({\mathcal R}_{N})_{\eps}} \, \mathring\cup \,  \bRCD X \, \mathring\cup  \, (\S^{N-2}\setminus (\bRCD X\cup \overset{\circ}{({\mathcal R}_{N})_{\eps}})).
$$

In words:   $X$  is the disjoint union of a manifold part of dimension $N$, a boundary part of Hausdorff dimension at most $N-1$, and a singular set of Hausdorff dimension at most $N-2$. 
\end{theorem}

\begin{proof}
We first show that $x\notin \S^{N-1} \Rightarrow x\notin \bRCD X$:  If $x\notin \S^{N-1}$ then there exists a tangent space to $x$ isomorphic to $\R^{N}$.  By Bishop-Gromov monotonicity it follows that \emph{every} tangent space to $x$ is isomorphic to $\R^{N}$, and thus $x\notin \bRCD X$.  

Since  $\bRCD X\subset \S^{N-1}$, \eqref{H-dim-Sk}  implies  that $ \bRCD X$  has Hausdorff dimension at most $N-1$.

\textbf{Proof of \eqref{regular-part} .} 
We claim that $({\mathcal R}_{N})_{\eps}\subset  \overset{\circ}{({\mathcal R}_{N})}_{\Psi(\eps|N)}$.
\\ {Combining \eqref{eq:tanRN} with Theorem~\ref{thm:ContHN}}, it holds that $|\vol(B_{1}^{Y}(y))-\omega_n|\le \Psi(\eps|N)$. Therefore $1\ge \vartheta_N(x)\ge 1- \Psi(\eps|N)$. By semicontinuity of $\vartheta_N$ this implies that $1\ge \vartheta_N(z)\ge 1- \Psi(\eps|N)$ for all $z$ sufficiently close to $x$. By Corollary \ref{cor:RNepsTheta} this implies that all $z$ near $x$ belong to ${({\mathcal R}_{N})}_{\Psi(\eps|N)}$, i.e.  $({\mathcal R}_{N})_{\eps}\subset  \overset{\circ}{({\mathcal R}_{N})}_{\Psi(\eps|N)}$ as claimed. By Theorem~\ref{thm:topologyHN} this implies \eqref{regular-part} as soon as $\eps_0$ is small enough so that for all $0<\eps<\eps_0$ it holds that 
$\Psi(\eps|N)< \bar\varepsilon(K,N,\alpha)$ given by Theorem~\ref{thm:topologyHN}.
\\

\textbf{Proof of \eqref{regular-no-bry}.} 
We argue by induction on $N$. The base of induction $N=1$ is easy due to the classification of $\RCD(K,1)$ spaces.

Suppose the statement holds for $N-1\ge 1$ and we need to prove it for $N$.

As above,  if $x \in  ({\mathcal R}_{N})_{\eps}$ then for every $(Y,y)\in {\rm Tan}(X,x)$ it holds  \eqref{eq:tanRN}.
On the other hand, if $x\in \bRCD X$, then there exists $(\bar{Y},\bar{y})\in {\rm Tan}(X,x)$ such that $Y=C(Z)$ where $Z$ is a non-collapsed $\RCD(N-2,N-1)$ space with boundary.
\\The estimate \eqref{eq:tanRN} implies that 
\begin{equation}\label{eq:ZRN}
\sfd_{GH}(Z, {\mathbb S}^{N-1})\leq \Psi(\eps|N).
\end{equation}
From the Sphere Theorem \ref{thm:Sphere}, we infer that $Z$ has almost maximal volume, i.e.  $\cH^{N-1}(Z)\geq (1-\Psi(\eps|N-1)) \cH^{N-1}( {\mathbb S}^{N-1})$. By the Bishop-Gromov monotonicity of volumes, it follows that for every $z\in Z$ it holds
$$
\lim_{r\to 0} \frac{\cH^{N-1}(B_{r}(z))}{\omega_{N-1} r^{N-1}}\geq 1-\Psi(\eps|N-1).
$$
{ If  $\eps\le \bar{\eps}(N)$, Corollary \ref{cor:RNepsTheta} implies} that  $Z=({\mathcal R}_{N-1})_{\eps_0(N-2,N-1)}(Z)$ and therefore $\bRCD Z=\emptyset$ by the induction assumption. This is a contradiction and hence $({\mathcal R}_{N})_{\eps}\cap \bRCD X=\emptyset$.
\\

\textbf{Proof of \eqref{reg-intr-sing}.} Notice that, from the very definition of singular set and reduced boundary, and from  Lemma \ref{lem:TanSpXNC} we get
\begin{equation}\label{SdeX}
\S\setminus \S^{N-2}=\S^{N-1}\setminus \S^{N-2}=\bRCD^* X\subset \bRCD X.
\end{equation}
Thus $$ ({\mathcal R}_{N})_{\eps}\cap (\S\setminus \S^{N-2})\subset ({\mathcal R}_{N})_{\eps}\cap  \bRCD X=\emptyset,$$ where in the last identity we used \eqref{regular-no-bry}.
We conclude that $({\mathcal R}_{N})_{\eps}\cap \S\subset \S^{N-2}$. 
\\ {In particular, by \eqref{H-dim-Sk},}   $({\mathcal R}_{N})_{\eps}\cap \S$ has Hausdorff dimension at most $N-2$.
\\

\textbf{Proof of \eqref{reg-connected}.}  

{ First note that, by \eqref{H-dim-Sk},  $\cH^{N-1}(S)=0$ if and only if $\cH^{N-1}(\bRCD^* X)=0$.}

Let $0<\eps<\eps_0$.
By part \eqref{regular-part}  we know that $B=X\backslash \overset{\circ}{({\mathcal R}_{N})_{\eps}}\subset \S$. By the assumption this implies that $\mathcal H^{N-1}(B)=0$. Also, $B$ is obviously closed. 

Let $x,y\in X\backslash B= \overset{\circ}{({\mathcal R}_{N})_{\eps}}$. Then for any small $\delta>0$ the ball $B_\delta(y)$ lies in $\overset{\circ}{({\mathcal R}_{N})_{\eps}}$. By Corollary~\ref{non-collapsed-alm-all-geod} 
 there is $y'\in B_\delta(y)$ and a shortest geodesic $[x,y']$ which is entirely contained in  $\overset{\circ}{({\mathcal R}_{N})_{\eps}}$. Then the concatenation of $[xy']$ and any shortest $[y'y]$ lies in  $\overset{\circ}{({\mathcal R}_{N})_{\eps}}$ and has length $\le \sfd(x,y)+2\delta$. Since this holds for all small $\delta$ this proves  \eqref{reg-connected}.

\end{proof}

We suspect that the condition that $\mathcal H^{N-1}(\S)=0$ in part \eqref{reg-connected}  of Theorem~\ref{thm-structure} is not needed.
\begin{conjecture}
 Let $(X,\sfd,\cH^{N})$ be a non-collapsed $\RCD(K,N)$ space. Then  $\overset{\circ}{({\mathcal R}_{N})_{\eps}}$ is path connected for all small $\eps$.
 
 \end{conjecture}
Note, that this is known to be true for Alexandrov spaces: it follows from a result of Petrunin, stating that for Alexandrov spaces tangent  spaces are isometric along interiors of geodesics ~\cite{petruninsecondvariation}.

\section{Boundary and convergence}\label{Sec:BoundConv}

In~\cite[Theorem 6.1]{CC97} Cheeger and Colding proved that the limit of a non-collapsing sequence of $N$-manifolds with Ricci curvature bounded below satisfies the property that the singular set
$\S$ is contained in $\S^{N-2}$; following our terminology, the limit space  has  empty reduced boundary. We show that this theorem has the following natural generalization to non-collapsed $\RCD$ spaces.

\begin{theorem}\label{thm-no-bry-main}
Let \{$(X_{i}, \sfd_{i},  \cH^{N})\}_{i\in \N}$ be a sequence of non-collapsed $\RCD(K,N)$ spaces. Assume that
\begin{itemize}
\item $\{(X_{i}, \sfd_{i}, p_{i})\}_{i\in \N}$ converge to  $(X, \sfd, p)$ in pointed Gromov-Hausdorff sense.
\item The open ball $B_1(p_i)$ is a topological $N$-manifold and $\cH^{N}(B_1(p_i))\ge v>0$ for all $i$.
\end{itemize}
Then $(X, \sfd, \cH^{N})$ is a non-collapsed $\RCD(K,N)$ space with  $\bRCD X \cap B_1(p)=\emptyset$.
\\ In particular $\S(X)\cap B_1(p)\subset \S^{N-2}(X)$ by Lemma~\ref{lem:TanSpXNC}.
\end{theorem}

{We will prove Theorem \ref{thm-no-bry-main} later in the section, let us first draw some consequences.} Applying this theorem to a constant sequence immediately gives

\begin{corollary}\label{manifold-no-boundary}
Let $(X,\sfd,\cH^{N})$ be a non-collapsed $\RCD(K,N)$ space. 
\\Suppose $(X,\sfd)$ is a topological $N$-manifold with boundary $\partial X$. Then  $\bRCD X \subset \partial X$.\\In particular if   $X$ is a topological $N$-manifold without boundary in manifold sense, then it is also without boundary in the $\RCD$ sense.

\end{corollary}
{
We don't know if in the above Corollary the inclusion $\bRCD X \subset \partial X$ is always an equality.  This is known to be true if $X$ is an Alexandrov space with curvature bounded below~\cite{Per-stab}.
We can also prove it in the following special case.

\begin{corollary}\label{cor:bilip}
Let $(X,\sfd,\cH^{N})$ be a non-collapsed  $\RCD(K,N)$ space which is  bi-Lipschitz homeomorphic to a smooth $N$-manifold with boundary. 
Then the $\RCD$ boundary of $X$ agrees with the manifold boundary, i.e. $\bRCD X=\partial X$.
\end{corollary}
\begin{proof}
Suppose $f\co M^N\to X$ is an $L$-bi-Lipschitz homeomorphism where $M$ is a smooth Riemannian manifold with boundary, 
$X$ is a non-collapsed $\RCD(K,N)$ space and $L>0$. By Corollary~\ref{manifold-no-boundary} it holds that 
$\bRCD X\subset f(\partial M)$.
Suppose there is $p\in\partial M$ such that  \emph{some} tangent space $(T_{f(p)}X,o)=\lim_{r_j\to 0} \frac{1}{r_i}(X,f(p))$ does not have boundary.   
Looking at $f\co \frac{1}{r_i}(M,p)\to  \frac{1}{r_i}(X,f(p))$  by Arzela-Ascoli's Theorem we can pass to a subsequence and get a limit map (which can be thought of as ``a differential'' of $f$ at $p$) $f_0\co T_pM=\R^N_+\to T_{f(p)}X=C(Z_0)$ which is also an $L$-bi-Lipschitz homeomorphism.{   Here we use the short-hand notation $T_{f(p)}X=C(Z_0)$ to denote a tangent space, without any claim of uniqueness; moreover, we will use the suggestive notation $f_0=d_{p} f\co T_pM=\R^N_+\to T_{f(p)}X=C(Z_0)$ without any claim of differentiability of $f$ at $p$, but just to stress that $f_{0}$ is a blow up of $f$ at $p$.}
Observe that $C(Z_0)$ is  a non-collapsed $\RCD(0,N)$ space and a metric cone, and $Z_0$ is a noncollapsed $\RCD(N-2,N-1)$ space with $\bRCD Z_0=\emptyset$.

 Let $p_0\in \partial \R^N_+$ be any point different from the origin { and consider $q=f_0(p_0)$}.
Then clearly $q$ is not the vertex of the cone $C(Z_0)$, i.e. it has the form  $q=(t_0,z_0)$ where $z_0\in Z_0$ and $t_0>0$. 

 {Repeating the same blow-up procedure for $f_0: \R^N_+\to C(Z_{0})$ at $p_0$ we obtain ``a differential''} $f_1=d_{p_0}{f_0}\co  \R^N_+\to T_{q}C(Z_0)=\R\times C(Z_1)$ where $C(Z_1)=T_{z_0}Z_0$  is  a noncollapsed $\RCD(0,N-1)$ space and a metric cone and $Z_1$ is a noncollapsed $\RCD(N-3,N-2)$ space with $\bRCD Z_1=\emptyset$.  Proceeding by induction for any $k\le N-1$, we can construct bi-Lipschitz homeomorphisms $f_k\co \R^N_+\to \R^k\times C(Z_k)$ where $C(Z_k)$ is a noncollapsed $\RCD(0,N-k)$ space without boundary.

 Indeed, suppose $k<N-1$ and we have already constructed $f_k$. Since $f_k$ is bi-Lipschitz and $k<N-1$, we can find  $p_k$ in $\partial \R^N_+$ such that { $f_{k}(p_k)\notin \R^k\times \{o\}$}, i.e. $f_{k}(p_k)=(x_k, t_k, z_k)$ with
$x_k\in \R^k, t_k>0$ and $z_k\in Z_k$. Then we can set $f_{k+1}$ { to be a blowup of $f_{k}$ at $p_{k}$, i.e. $$f_{k+1}=d_{p_{k}}{f_k}:T_{p_k}\R^N_+=\R^N_+\to T_{f_{k}(p_k)} (\R^k\times C(Z_k))=\R^{k+1}\times C(Z_{k+1})$$} where $ C(Z_{k+1})=T_{z_k}Z_k$.

 On the last step  we get a map  $f_{N-1}\co \R^N_+\to \R^{N-1}\times C(Z_{N-1})$ where $C(Z_{N-1})$ is a noncollapsed  $\RCD(0,1)$ space without boundary and a metric cone. By the classification of $\RCD(0,1)$ spaces this can only be $\R$. This means that we have a homeomorphism $f_{N-1}\co \R^N_+\to \R^{N}$. This is  impossible and therefore $f(\partial M)=\bRCD X$.
\end{proof}

\begin{remark}
It's easy to see that the proof of Corollary~\ref{cor:bilip} works more generally if, instead of assuming that $X$ is bi-Lipschitz to a smooth manifold, we assume that $X$ is a Lipschitz manifold with boundary and the metric $\sfd$ is compatible with the Lipschitz structure on $X$. In other words, if $X$ admits an atlas of charts which are bi-Lipschitz maps to open subsets of $\R^N_+$.
\end{remark}

\begin{remark}
The proof of Corollary~\ref{cor:bilip} shows that for $X$ satisfying the assumptions of the Corollary, the answer to Question ~\ref{q:boundTan} is positive; i.e. a point $p\in X$ belongs to $\bRCD X$ if and only if \emph{every} tangent space $T_pX$ has  boundary.
\end{remark}

As it was  suggested  to the authors by Alexander Lytchak,  using a similar blow up argument Corollary \ref{cor:bilip}   implies the following stronger result.

\begin{proposition}\label{bilip-preserve-bry}
Let $(X,\sfd_{X},\cH^{N})$ and  $(Y,\sfd_{Y},\cH^{N})$ be non-collapsed  $\RCD(K,N)$ spaces. Assume there is a bi-Lipschitz homeomorphism $f:X\to Y$. Then $f(\bRCD X)=\bRCD Y$.
\end{proposition}

\begin{proof}
Clearly, since $f^{-1}:Y\to X$ is also a bi-Lipschitz homeomorphism, it is enough to show that $f(\bRCD X)\subset \bRCD Y$. We argue by contradiction. If it is not the case, then there exists  $p\in \bRCD X$ (i.e. there exists a tangent space $T_{p}X=C(X_0)$ with non-empty $\RCD$-boundary) such that  every tangent space $T_{f(p)} Y$ at $f(p)\in Y$ has empty $\RCD$-boundary. As in the proof of Corollary \ref{cor:bilip}, there is no claim of uniqueness of tangent  { spaces}, we use the shorthand notation $T_{p}X$ just for convenience.

Also, again as in the proof of Corollary \ref{cor:bilip} we can  pass to a subsequence and get a limit  bi-Lipschitz ``blow up" map  $f_0=d_{p} f\co T_pX=C(X_{0})\to T_{f(p)}Y=:Z_0$. Pick a point $p_0\in \bRCD C(X_0)$ different from the origin ({ whose existence follows directly from Definition \ref{def:bRCD}}). Then $p_0=(t_0, x_0)$ where $t_0>0$ and $x_0\in \bRCD X_0$. Recall that $\bRCD Z_0=\emptyset$ and hence $f_0(p_0)$ is not a boundary point. Next we can take
$f_1=d_{p_0}f_0\co T_{p_0}C(X_0)\to T_{f_0(p_0)}Z_0=: Z_1$ which is a bi-Lipschitz homeomorphism and $T_{p_0}C(X_0)$ has boundary while $Z_1$ does not. Note that, by the Splitting Theorem {\cite{giglisplittingshort}}, $T_{p_0}C(X_0)\cong \R\times C(X_1)$ where $C(X_1)=T_{x_0}X_0$ is non-collapsed $\RCD(N-1,0)$ { with $\bRCD C(X_{1})\neq \emptyset$}.  We can iterate this construction further to get, for $k=0,\ldots, N-1$, bi-Lipschitz homeomorphisms $f_k\co \R^k\times C(X_k)\to Z_k$ where each $Z_k$ is a non-collapsed $\RCD(0,N)$ space without boundary and each $C(X_k)$ is a non-collapsed $\RCD(0,N-k)$ space with boundary. On the very last step the space $C(X_{N-1})$ is a non-collapsed $\RCD(0,1)$ space with boundary which can only happen if  $C(X_{N-1})\cong [0,\infty)$,  by \cite{KL}.
Therefore $f_{N-1}\co \R^N_+\to Z_{N-1}$ is a bi-Lipschitz homeomorphism and $\bRCD(Z_{N-1})=\emptyset$. Since $\R^N_+$ is a smooth manifold with boundary this is impossible by Corollary \ref{cor:bilip}. Therefore $f(\bRCD X) \subset \bRCD Y$.

\end{proof}

Theorem~\ref{thm-no-bry-main}
also immediately implies the following result.

\begin{corollary}\label{thm-no-bry}

Let $\delta=\delta(N)$ be small enough so that $\eps(\delta,N)$ provided by Corollary~\ref{cor-reg} satisfy 
\[
\eps(\delta,N)<\beps(N,1/2),
\]
{ where $\beps(N,1/2)$ was given in Theorem \ref{thm:CCReif}}.
Then the following holds.

Let \{$(X_{i}, \sfd_{i}, \cH^{N})\}_{i\in \N}$ be a sequence of non-collapsed $\RCD(K,N)$ spaces. Assume that
\begin{itemize}
\item $\{(X_{i}, \sfd_{i}, p_{i})\}_{i\in \N}$ converge to  $(X, \sfd, p)$ in pointed Gromov-Hausdorff sense.
\item  $X_i\cap B_1(p_i)\subset (\cWR_N)_\eps(X_i)$   and $\cH^{N}(B_1(p_i))\ge v>0$ for all $i$.
\end{itemize}
Then $(X, \sfd, \cH^{N})$ is a non-collapsed $\RCD(K,N)$ space with  $\bRCD X \cap B_1(p)=\emptyset$.
\end{corollary}

\begin{proof}[Proof of Corollary~\ref{thm-no-bry}]
Observe that by Corollary~\ref{cor-reg}, all $X_i$ satisfy $X_i\cap B_1(p_i)\subset (\cR_{N})_{\beps}(X_i)$ for all $i$ and hence all   $X_i\cap B_1(p_i)$ are topological $N$-manifolds by the Cheeger-Colding-Reifenberg Theorem \ref{thm:CCReif}. Now the result follows by Theorem~\ref{thm-no-bry-main}.
\end{proof}

For the proof of Theorem~\ref{thm-no-bry-main}  we will need the following well-known folklore result in topology.

\begin{theorem}\label{thn-comp-exhaustion}
Let $M^n$ be a  connected non-compact topological manifold. Then $M$ admits an exhaustion  $K_0\subset K_1\subset\ldots$
by compact connected $n$-dimensional submanifolds  $K_i$ with boundary such that $\cup_i K_i=M$ and $K_i\subset \intr K_{i+1}$ for all $i$.
\end{theorem}
\begin{proof}
Since we don't know an explicit reference to this statement in literature we briefly sketch the argument from known results.

For $n\le 3$ all topological $n$-manifolds are smoothable and for smooth manifolds the statement easily follows by taking an exhaustion by regular sublevel sets of a smooth proper function.

In dimension $4$ it was proved by Quinn that any noncompact connected manifold is smoothable \cite{QuinnIII} hence the same argument applies.
In dimensions $\ge 6$ the result immediately follows from work of Kirby and Siebenmann ~\cite{KS} who proved that for $n\ge 6$ all topological manifolds admit handle decompositions.
Existence of handle decompositions was proved by Freedman and Quinn \cite{QuinnIII} for $n=5$ which gives the proof in that dimension also.
\end{proof}

\begin{proof}[Proof of Theorem ~\ref{thm-no-bry-main}]
We first observe that $(X,\sfd,\cH^{N})$ is a non-collapsed $\RCD(K,N)$ space: the fact that $(X_{i},\sfd_{i},\cH^{N},p_{i})$ are $\RCD(K,N)$ implies by Gromov's compactness theorem that they converge (up to subsequences) to a limit p.m.m.s. $(Y,\sfd_{Y},\mm,\bar{y})$ in the pointed measured Gromov Hausdorff sense. Since by assumption $(X_{i},\sfd_{i},p_{i}) \to (X,\sfd, p)$ in pmGH sense, then $(X,\sfd, p)$ is isometric to $(Y,\sfd_{Y},y)$.  By the stability of $\RCD(K,N)$ under pmGH convergence \cite{lottvillani:metric, sturm:II, Vil, AGS, GMS2013} it follows that  $(X,\sfd,\mm)$ is  $\RCD(K,N)$. The  assumption $\cH^{N}(B_1(p_i))\ge v>0$ implies that $(X,\sfd,\mm)$ is  a non-collapsed $\RCD(K,N)$ space by~\cite[Theorem 1.2]{GDPNonColl}, i.e. $\mm=\cH^{N}$.

Since all the spaces involved are non-collapsed $\RCD(K,N)$, the background measure is always the $N$-dimensional Hausdorff measure.  We will therefore suppress the measure in notations for $\RCD$ spaces occurring in the proof.

Suppose by contradiction that $\bRCD X\ne \emptyset$.  { From Lemma \ref{lem:bRCDXIteratedTang},} if $x\in\bRCD X$ then some iterated tangent space $T_{\xi_N}T_{\xi_{N-1}}\ldots T_{\xi_1} X$ is isometric to 
$\R^N_+=\{(x_1,\ldots,x_N)| x_1\ge 0\}$. Here $x=\xi_1\in X, \xi_2\in T_{\xi_1}X$ etc.

Next note that if $(Y_i, \sfd_i, \hat y_i,)\to (Y, \sfd, \hat y)$ is a pointed GH-converging sequence of  spaces then for any $y\in Y$ and any tangent space $T_yY$ by a diagonal argument there exists a sequence of rescalings $\lambda_k\to\infty$, a subsequence $Y_{i_k},$ and a sequence of base points $y_{i_k}$ such that $(Y_{i_k},\lambda_k d_{i_k},y_{i_k})\to (T_yY,o)$ as $k\to\infty$.

Combining the above observations implies that after passing to a subsequence, up to a change of base points and rescalings we can assume that
to begin with $(X_i,p_i)\to (\R^N_+,p)$ where $p=0$ and $X_i$ is $\RCD(K_i,N)$ with $K_i\to 0$.

Let $f_i\co B_1(p_i)\to B_1(0)\cap \R^N_+$ be a $\delta_i$- Gromov-Hausdorff approximation with $\delta_i\to 0$. By a standard partition of unity center of mass argument we can assume that $f_i$ is continuous.
Namely,  let $g_i\co B_1(p_i)\to B_1(0)\cap \R^N_+$ be a $\delta_i$-GH approximation. Take a maximal finite $\delta_i$-separated net $\{x_1,\ldots, x_m\}$ in  $B_1(p_i)$. Then the balls $\{B_{\delta_i}(x_j)\}_{j=1}^m$ cover $B_1(p_i)$. Let $\lambda_j(x)$ be a partition of unity subordinate to this cover.

Set $f_i(x):=\sum_j\lambda_j(x)g_j(x_j)$. Then $f_i$ is continuous and uniformly $10\delta_i$ close to $g_i$.

Since  all $X_i$ are topological manifolds, by   Theorem~\ref{thn-comp-exhaustion} there exist compact connected submanifolds with boundary $K_i\subset B_1(p_i)$ such that $\bar B_{1-\delta_i}(p_i)\subset K_i$.

Since $f_i$ is $\delta_i$-GH approximation and $\partial B_1(0)\cap \R^N_+$ is exactly the unit sphere around $0$ in $\R^N_+$  we must have that 
$f_i(\partial K_i)$ is contained in the $2\delta_i$-neighborhood of $\partial B_1(0)\cap \R^N_+$. By adjusting the maps $f_i$ (along radial projections in $\R^N$) we can assume that

\begin{equation}\label{bry-incl}
f_i(\partial K_i)\subset \partial B_1(0)\cap \R^N_+.
\end{equation}

Note that $\partial B_1(0)\cap \R^N_+$ is a proper submanifold of codimension 0 homeomorphic to $\bar D^{N-1}$ in $\partial (B_1(0)\cap \R^N_+)\cong \Ss^{N-1}$ and the same holds for $\bar B_1(0)\cap \partial \R^N_+$.

Let $q=(1/2,0,\ldots 0)\in \R^N_+$ and let $q_i\in X_i$ be such that $\sfd(q, f_i(q_i))\le \delta_i$. Obviously, such $q_i$ exists. By modifying the map $f_i$ slightly by a post-composition with a self homeomorphism of $\R^N_+$ which is identity outside $B_{0.49}(q)$  we can assume that $f_i(q_i)=q$.

By the Topological Stability Theorem \ref{top-stability-RCD} for all large $i$ there exist topological embeddings $h_i\co B_{1/3}(q_i)\to B_{1/3+\eps_i}(q)$ with $\eps_i\to 0$ which are also $\eps_i$-GH approximations and such that $h_i(B_{1/3}(q_i))\supset B_{1/3-10\eps_i}(q)$.

By a straight line interpolation we can change $f_i$ slightly  to a $\Psi(\delta_i,\eps_i|N)$-close map $\hat f_i$ such that $\hat f_i=h_i$ on $B_{1/5}(q_i)$ and $\hat f_i=f_i$ outside $B_{1/3}(q_i)$. Indeed, let $\lambda\co\R\to \R$ be a continuous function with $0\le\lambda\le 1$, 
$\lambda(x)=0$ for $x\le 1/5$ and $\lambda=1$ for $x\ge 1/4$. Then  $\hat f_i(x)=\lambda (\sfd(x,q))f_i(x)+(1-\lambda (\sfd(x,q)))h_i^{-1}(x)$ works.

Now, let $M_i$ be the double of $K_i$ along its boundary and let $M$ be the double of $ \bar B_1(0)\cap \R^N_+$ along $\partial B_1(0)\cap \R^N_+$. Note that $M$ is topologically a closed disk $\bar D^N$ and $M_i$ is a connected closed manifold without boundary.

By~\eqref{bry-incl} we can ``double"  $\hat f_i$ along $\partial K_i$ and extend it to a map $\tilde f_i\co M_i\to M$. Then if we compute $\Z_2$ degree of $\tilde f_i$ on the one hand it must be zero since $M$ is not a closed manifold. On the other hand it must be equal to 1 for large $i$ since $\tilde f_i$ is a homeomorphism on $B_{1/5}(q_i)$ and $q$ has a unique preimage under $\tilde f_i$ and this preimage is contained in $B_{1/5}(q_i)$. This is a contradiction and hence $\bRCD X=\emptyset$.
\end{proof}

Next we will show that  Theorem~\ref{thm-no-bry-main}  still holds if the elements of the sequence are allowed to have more severe singularities provided the singular set is reasonably small.

We will  make use of the following generalization to  non-collapsed $\RCD(K,N)$ spaces due to Antonelli, Bru\'e and Semola~\cite{Ant-Brue-Sem-19}  of the quantitative stratification result of Cheeger and Naber ~\cite{CheNab13}, originally proved for smooth Riemannian manifolds with Ricci and volume bounded below.

\begin{theorem}[Quantitative Stratification, \cite{CheNab13,Ant-Brue-Sem-19}]\label{thm-quant-str-1}
Given $v>0$, $\eps>0$, $0<\eta<1$ and non-negative integers $k<N$ there exists $c(k,K,N,v,\eps,\eta)$ such that if $(X,\sfd, \cH^{N})$ is a non-collapsed $\RCD(K,N)$ space with $\cH^{N} (B_1(p))\ge v$ then

\[
\cH^{N} \left(B_r(\S^k_{\eps,r}\cap B_1(p))\right)\le c(k,K,N,v,\eps,\eta)r^{N-k-\eta}, \quad \forall r\in (0,1).
\]
\end{theorem}

We will prove:

\begin{theorem}\label{thm-no-bry-2}
For any $K\in \R$ and $N\in \N$ there exists $\hat \eps(K,N)$  such that the following holds.

Let \{$(X_{i}, \sfd_{i}, \cH^{N}_{\sfd_{i}})\}_{i\in \N}$ be a sequence of non-collapsed $\RCD(K,N)$ spaces. Assume that
\begin{itemize}
\item $\{(X_{i}, \sfd_{i}, p_{i})\}_{i\in \N}$ converge to  $(X, \sfd, p)$ in pointed Gromov-Hausdorff sense.
\item  $\cH^{N}(B_1(p_i))\ge v>0$ for all $i\in \N$.
\item  For any $i\in \N$ it holds that $\hat \S_{\hat \eps}^{N-1}(X_i)\subset \S^{N-2}_{\hat\eps}(X_i)$.
\end{itemize}
Then $(X, \sfd, \cH^{N})$ is a non-collapsed $\RCD(K,N)$ space with  $\bRCD X=\emptyset$. 
\end{theorem}

We believe that the following natural conjecture should hold in general.
\begin{conjecture}\label{conj-no-bry-stable}
Let \{$(X_{i}, \sfd_{i}, \cH^{N}_{\sfd_{i}})\}_{i\in \N}$ be a sequence of non-collapsed $\RCD(K,N)$ spaces. Assume that
\begin{itemize}
\item $\{(X_{i}, \sfd_{i}, p_{i})\}_{i\in \N}$ converge to  $(X, \sfd, p)$ in pointed Gromov-Hausdorff sense.
\item  $\cH^{N}(B_1(p_i))\ge v>0$ for all $i\in \N$.
\item  $\bRCD X_i=\emptyset$ for all $i$.
\end{itemize}
Then $(X, \sfd, \cH^{N})$ is a non-collapsed $\RCD(K,N)$ space with  $\bRCD X=\emptyset$
\end{conjecture}

\begin{remark}
The corresponding statement is known to be true for Alexandrov spaces.  This follows from Perelman's Stability Theorem but can also be proved by more elementary methods similar to the proofs of Theorems \ref{thm-no-bry-main} and ~\ref{thm-no-bry-2}.

Also note that if the answer to Question ~\ref{bry-reduced-bry} is positive then Conjecture~\ref{conj-no-bry-stable} is equivalent to conjecturing that if $\S^{N-1}(X_i)\subset \S^{N-2}(X_i)$ then $\bRCD X=\emptyset$.
\end{remark}

It is also natural to ask the opposite question.
\begin{question}
 Suppose  $(X_i,\sfd_{i},p_i)\to (X,\sfd, p)$  is a converging sequence of non-collapsed $\RCD(K,N)$ spaces with $\cH^{N}(B_1(p_i))\ge v>0$ for all $i$. Suppose further that
$\bRCD X_i\cap B_1(p)\ne \emptyset$ (resp. $\bRCD^*  X_i\cap B_1(p)\ne \emptyset$). Does this imply that $\bRCD X\ne \emptyset$ (resp. $\bRCD^* X \ne \emptyset$) as well? This again is known for Alexandrov spaces by Perelman's Stability Theorem.
\end{question}

\begin{proof}[Proof of Theorem~\ref{thm-no-bry-2}]
Let $\alpha=1/2$ and let $\beps=\beps(N,\alpha)$ be the constant provided by { Theorem \ref{thm:CCReif}}. Further let $\delta>0$ be such that $\eps(\delta,K,N)$ given by  Corollary~\ref{cor-reg} is smaller than $\beps$.

Finally, set $\hat\eps=\min \{\delta, \beps\}$. We claim that $\hat\eps$ satisfies the conclusions of the theorem. Suppose $(X_i,p_i)\to (X,p)$ is a contradicting sequence.

Recall that by  Remark~\ref{quant-scale} the inclusion $\hat S^{N-1}_{\hat\eps}(X_i)\subset S^{N-2}_{\hat\eps}(X_i)$ remains true after rescaling the metric by any $\lambda\ge 1$.
Therefore,  as in the proof of Corollary~\ref{thm-no-bry} we can assume that $(X,p)=(\R^N_+,0)$ and $X_i$ is non-collapsed $\RCD(K_i,N)$ with $K_i\to 0$.

Let $f_i\co (B_1(p_i),0)\to (B_1(0)\cap \R^N_+,0)$ and $h_i\co  (B_1(0)\cap \R^N_+,0)\to (B_1(p_i),0)$ be $\delta_i$-GH-approximations with $f_i\circ h_i$ and   $h_i\circ f_i$ both $\delta_i$-close to identity.

Let $\eta=1/2, k=N-2$ and let  $c(k,-1, N,v,\hat\eps,\eta)$ be given by the Quantitative Stratification Theorem \ref{thm-quant-str-1} where $v=\omega_N/10$.

Fix an $r>0$ be small enough so that $100c(N-2,-1,N,v,\hat\eps,\eta)r^{3/2}\le \omega_{N-1}r$.
\\Let $U_{i,r}=B_r(h_i(\partial \R^N_+))\cap B_1(p_i)$. Then by volume continuity (Theorem~\ref{thm:ContHN}) we have that 
\[
\cH^{N}( U_{i,r})\ge \frac{\omega_{N-1}}{2}r.
\]
On the other hand, by Theorem~\ref{thm-quant-str-1} we have that
\[
\cH^{N} \left( B_{3r}(S^{N-2}_{\hat\eps,3r})\cap B_1(p)\right)\le c(N-2,-1,N,v,\hat\eps,\eta)(3r)^{3/2}<\frac{\omega_{N-1}}{2}r.
\]
Therefore there exists $\hat q_i\in U_{i,r}\backslash B_{3r}(S^{N-2}_{\hat\eps,3r})$.

By above $B_{3r}(\hat q_i)$ contains no points from $S^{N-2}_{\heps,3r}$. We can find $q_i\in B_{3r}(\hat q_i)$ such that  $B_r(q_i)$ contains no points from $S^{N-2}_{\heps,3r}$ and $\sfd(f_i(q_i), \partial \R^N_+)\le \Psi(\delta_i)$.
\\After passing to a subsequence we can assume that $(X_i,q_i)\to (\R^N_+,q)$ with $q\in \partial \R^N_+$.

By the assumptions of the theorem the above implies that $B_r(q_i)\cap \hat S^{N-1}_{\hat \eps}=\emptyset$.
By the definition of $\hat S^{N-1}_\eps$ this means that for any $x\in B_r(q_i)$ \emph{some} tangent space $T_xX_i$ is $\hat \eps$ close to $\R^N$. Since $\hat \eps\le\delta$, by Corollary~\ref{cor-reg} \emph{any} tangent space $T_xX_i$ is $\beps(N)$-close to $\R^N$. This means that $B_r(q_i)$ satisfy the regularity assumptions in Corollary~\ref{thm-no-bry} for all large $i$. Now Corollary~\ref{thm-no-bry} implies the result.

\end{proof}

In \cite{KaLytPe} a different notion of a boundary, called metric-measure boundary or mm-boundary was introduced. 
\begin{question}
What is the relation between $\bRCD X$ and the mm-boundary of $X$? In particular, is it true that if  $\bRCD X=\emptyset$ then the mm-boundary of $X$ is zero? Is the same true if we only assume that $\bRCD^*X=\emptyset$?
\end{question}

\section{``Sequential openness'' of weakly non-collapsed $\RCD(K,N)$ spaces}\label{Sec:SeqOpenNCRCD}

Following the terminology proposed in \cite{GDPNonColl}, we say  that an $\RCD(K,N)$ space $(X,\sfd,\mm)$ is  \emph{weakly non-collapsed} if $\mm\ll \cH^{N}$. It was recently proved by Honda \cite{HondaNC} that a compact weakly non-collapsed $\RCD(K,N)$ space $(X,\sfd,\mm)$ is non-collapsed (up to a constant rescaling of the measure), i.e. $\mm=c \cH^{N}$ for some constant $c>0$. 

The goal of the present section is to prove a series of results stating roughly that if the limit  of a pmGH sequence of $\RCD(K,N)$ spaces is  (weakly) non-collapsed, then the same is true  eventually for the elements of the sequence; thus establishing a sort of ``sequential openness'' of this class of spaces.

\begin{theorem}\label{thm:Stabwnc}
Let $(X,\sfd,\mm, \bar{x})$ be a pointed weakly non-collapsed $\RCD(K',N)$ space for some $K'\in \R, N\in \N$. Let $\{(X_{i}, \sfd_{i},\mm_{i}, \bar{x}_{i})\}_{i\in \N}$ be a sequence of pointed $\RCD(K,N)$ spaces, for some $K\in \R$,  converging to 
$(X,\sfd,\mm, \bar{x})$ in pointed measured Gromov Hausdorff sense. Then there exists $i_{0}\in \N$ such that $(X_{i}, \sfd_{i},\mm_{i})$ is a  weakly non-collapsed $\RCD(K,N)$ space for every $i\geq i_{0}$.
\end{theorem}

\begin{proof}
Without loss of generality we can assume that $\bar{x}$ is an $N$-regular point for $(X,\sfd,\mm)$. In particular, for every $\delta>0$ there exists $r=r(\bar{x}, \delta)$ such that
\begin{equation*}
\sfd_{mGH} \left( \left(B_{r}^{X}(\bar{x}), \sfd,  \frac{1} {\mm(B_{r}^{X}(\bar{x}))} \mm\llcorner B_{r}^{X}(\bar{x})\right), \left(B_{r}^{\R^{N}}(0^{N}),\sfd_{E},\frac{1}{\cL^{N}(B_{r}^{\R^{N}}(0^{N}))} \cL^{N} \llcorner (B_{r}^{\R^{N}}(0^{N}) \right) \right)\leq \delta r.
\end{equation*}
Since by assumption $(X_{i}, \sfd_{i},\mm_{i}, \bar{x}_{i})\to (X,\sfd,\mm, \bar{x})$ in pmGH sense,  there exists $i_{0}=i_{0}(\delta, {r})\in \N$  such that for all $i\geq i_{0}$:
\begin{equation}\label{eq:mGHBrxi}
\sfd_{mGH} \left( \left(B_{r}^{X_{i}}(\bar{x}_{i}), \sfd_{i},  \frac{1} {\mm_{i}(B_{r}^{X_{i}}(\bar{x}_{i}))} \mm_{i}\llcorner B_{r}^{X_{i}}(\bar{x}_{i})\right), \left(B_{r}^{\R^{N}}(0^{N}),\sfd_{E},\frac{1}{\cL^{N}(B_{r}^{\R^{N}}(0^{N}))} \cL^{N} \llcorner (B_{r}^{\R^{N}}(0^{N})) \right) \right)
\leq 2 \delta r.
\end{equation}
Combining \eqref{eq:mGHBrxi} with $\eps$-regularity \cite[Theorem 6.8]{MN} it follows that,  for all $i\geq i_{0}$, there exist a subset $U_{i}\subset B_{r}^{X_{i}}(\bar{x})$ with $\mm_{i}(U)>0$ and a  $(1+\eps)$-bi-Lipschitz  map $u_{i}:U_{i}\to u_{i}(U_{i})\subset \R^{N}$, where $\eps=\Psi(\delta|K,N)$. Moreover, from \cite[Proposition 3.2]{GDPNonColl}, it holds that $\cL^{N}(u_{i}(U_{i}))>0$. Since $u_{i}$ is $(1+\eps)$-bi-Lipschitz it follows that $\cH^{N}(U_{i})>0$.
\\From the rectifiability of $\RCD(K,N)$ spaces as metric measure spaces \cite[Theorem 1.2]{KM} (see also \cite{DePRinMar} and \cite{GiPa} for independent proofs), it follows that the regular stratum ${\mathcal R}_{N}(X_{i})$ of dimension $N$  satisfies:
\begin{equation}\label{eq:RN>0}
\mm_{i}({\mathcal R}_{N}(X_{i}))>0, \quad \mm_{i}\llcorner {\mathcal R}_{N}(X_{i})\ll \cH^{N}_{\sfd_{i}}.
\end{equation}
The constancy of the dimension in $\RCD(K,N)$ spaces proved in \cite{BS18} yields that 
\begin{equation}\label{eq:RCDconst}
\mm_{i}(X_{i}\setminus {\mathcal R}_{N}(X_{i}))=0.
\end{equation}
The combination of \eqref{eq:RN>0} with \eqref{eq:RCDconst} gives that $(X_{i}, \sfd_{i},\mm_{i})$ is a  weakly non-collapsed $\RCD(K,N)$ space for every $i\geq i_{0}$.
\end{proof}

{
\begin{remark}
The above theorem also follows from ~\cite{KitaPOTA} where it is proved that the geometric dimension of $\RCD(K,N)$ spaces is lower semicontinuous under pmGH convergence. This implies that  $\mathcal R_N(X_i)$ has positive measure for large $i$ which by the same argument as above using \cite[Theorem 1.2]{KM} yields the result. 
\end{remark}
}
\begin{theorem}\label{thm:Stabnc}
Let $(X,\sfd,\cH^{N})$ be a compact non-collapsed $\RCD(K',N)$ space for some $K'\in \R, N\in \N$. Let $\{(X_{i}, \sfd_{i},\mm_{i})\}_{i\in \N}$ be a sequence of $\RCD(K,N)$ spaces, for some $K\in \R$,  converging to 
$(X,\sfd,\mm)$ in measured Gromov Hausdorff sense. Then there exists $i_{0}\in \N$ such that $(X_{i}, \sfd_{i},\mm_{i})$ is a compact non-collapsed $\RCD(K',N)$ space for every $i\geq i_{0}$.
More precisely, there exists a sequence $c_{i}\to 1$ such that $\mm_{i}=c_{i} \cH^{N}_{\sfd_{i}}$.
\end{theorem}

\begin{proof}
From Theorem \ref{thm:Stabwnc} we have that  $(X_{i}, \sfd_{i},\mm_{i})$ is a  weakly non-collapsed $\RCD(K,N)$ space for every $i\geq i_{0}$. Moreover, since the limit space $X$ is compact,  from the definition of pmGH convergence we have that $X_{i}$ is compact as well, for large $i$. Since every compact weakly non-collapsed $\RCD$ space is actually non-collapsed \cite{HondaNC} up to rescaling the background measure by a constant, we infer that there exist constants $c_{i}>0$ such that $\mm_{i}=c_{i} \cH^{N}_{\sfd_{i}}$.

We now claim that $c_{i}\to 1$.
\\The mGH convergence of $(X_{i}, \sfd_{i},\mm_{i}=c_{i} \cH^{N}_{\sfd_{i}})$ to $(X,\sfd,\cH^{N})$ ensures that 
\begin{equation}\label{eq:ConvVolXi}
c_{i} \cH^{N}_{\sfd_{i}}(X_{i}) \to \cH^{N}(X).
\end{equation}
On the other hand, the GH convergence of  $(X_{i}, \sfd_{i})$ to $(X,\sfd)$ combined with the volume continuity Theorem \ref{thm:ContHN} (applied with  $R>\limsup \diam X_{i}$) yields that
\begin{equation}\label{eq:ConvHNXi}
\cH^{N}_{\sfd_{i}}(X_{i}) \to \cH^{N}(X).
\end{equation}
Putting together \eqref{eq:ConvVolXi} and  \eqref{eq:ConvHNXi} gives the claim $c_{i}\to 1$.
\end{proof}
{

Collecting some results of the paper with others in the literature we obtain the following theorem { (compare also with \cite{GDPNonColl, KitaPOTA, AHPT}).}

\begin{theorem}\label{noncol-equiv}
Let $(X,\sfd,\mm)$ be an $\RCD(K,N)$ space. Then the following are equivalent:
\begin{enumerate}
\item There exists   $p \in X$ and a tangent space $(Y,\sfd_{Y},\mm_{Y})\in \Tan(X,\sfd,\mm, p)$ with $\mm_{Y}\ll \cH^{N}_{\sfd_{Y}}$.
\item There exists a point $p\in X$ such that $\R^{N}\in \Tan(X,\sfd,p)$, i.e. $\R^{N}$ with Euclidean metric is a metric tangent space at $p$.
\item For $\mm$-a.e. $p\in X$ the tangent space at $p$ is unique and isomorphic as a m.m.s. to Euclidean $\R^{N}$ (endowed with the suitable rescaled measure).
\item   $(X,\sfd,\mm)$ is a weakly  non-collapsed $\RCD(K,N)$ space, i.e. $\mm\ll \cH^{N}$.
\end{enumerate}
If moreover $(X,\sfd)$ is compact, then all the above statements are equivalent to  $(X,\sfd,\mm)$ being a non-collapsed $\RCD(K,N)$ space up to constant a normalization of $\mm$, i.e. $\mm=c \cH^{N}$ for some constant $c>0$.
\end{theorem}

\begin{proof}
The final claim in case of compact  $(X,\sfd)$ is a direct consequence of \cite{HondaNC}.
\medskip

(1)$ \Rightarrow$ (4). Since $(Y,\sfd, \mm_{Y})$ is a tangent space of an $\RCD(K,N)$ space, then $Y$ is an $\RCD(0,N)$ space as well. The assumption then implies that $Y$ is  weakly non-collapsed $\RCD(0,N)$. Hence (4)  follows directly by applying Theorem \ref {thm:Stabnc} to the blow up sequence pmGH converging to  $(Y,\sfd, \mm_{Y})$.
\medskip

(4) $\Rightarrow$ (3) Follows by combining the next results:  $\mm$-a.e. $x\in X$ has unique tangent space which is isomorphic to a Euclidean space \cite{MN},  $\mm$-a.e. uniqueness of the dimension of Euclidean tangent spaces \cite{BS18},  on the $k$-regular stratum  it holds $\mm\ll \cH^{k}$ (\cite[Theorem 1.2]{KM};  see also \cite{DePRinMar} and \cite{GiPa} for independent proofs).  
\medskip

(3) $\Rightarrow$ (2) is trivial.
\medskip 

(2)$ \Rightarrow$ (1). By the compactness of $\RCD(-1,N)$ spaces, there exists a Radon measure $\mm_{\infty}$ on $\R^{N}$ with $\supp \mm_{\infty}=\R^{N}$ so that   $(\R^{N},\sfd_{E}, \mm_{\infty})$ is a metric measure tangent space. In particular,  $(\R^{N},\sfd_{E}, \mm_{\infty})$ verifies $\RCD(0,N)$. It follows from \cite[Corollary 8.2]{CaMoAIM} that $\mm_{\infty}\ll \cL^{N}$ and thus $(\R^{N},\sfd_{E}, \mm_{\infty})$ is a weakly non-collapsed $\RCD(0,N)$ space appearing as a tangent.

\end{proof}
}

The combination of  Theorem \ref{top-stability-RCD} and Theorem \ref{thm:Stabnc} gives the following result.

\begin{theorem}\label{thm:XiMman} 
Let $(M,g)$ be a compact Riemannian manifold of dimension $N$. {Let $\{(X_{i}, \sfd_{i},\mm_{i})\}_{i\in \N}$ be a sequence of $\RCD(K,N)$ spaces, for some $K\in \R$,  mGH converging to 
$(M,g)$.} Then there exists $i_{0}\in \N$ such that 
\begin{itemize}
\item $(X_{i}, \sfd_{i},\mm_{i})$ is a compact non-collapsed $\RCD(K,N)$ space for every $i\geq i_{0}$: more precisely, there exists a sequence $c_{i}\to 1$ such that $\mm_{i}=c_{i} \cH^{N}$.
\item $(X_{i}, \sfd_{i})$ is homeomorphic to $M$ via bi-H\"older homeomorphisms.
\end{itemize}
\end{theorem}

{
 We can now combine the results obtained so far to give a proof of Corollary \ref{eq:lambda1N=N}.}
}

{
\begin{proof}[Proof of Corollary \ref{eq:lambda1N=N}]
Assume by contradiction that it is not true. Then we can find $\varepsilon_{0}>0$ and  a sequence $\{(X_{i},\sfd_{i},\mm_{i})\}_{i\in \N}$ of $\RCD(N-1,N)$ spaces such that $\lambda_{j}(X_{i})\to N$ as $i\to \infty$ for every $j=1,\dots, N+1$, and such that one of the conclusions $(1)-(4)$ fail for  $\varepsilon=\varepsilon_{0}$ for all $i\in \N$.

 By Gromov's compactness Theorem, stability of $\RCD(N-1,N)$ and stability of the spectrum  \cite[Theorem 7.8]{GMS2013} under mGH convergence (as well as equivalence of pmG and mGH for uniformly doubling spaces, see \cite[Theorem 3.30 and Theorem  3.33]{GMS2013})  we get that  there exists an $\RCD(N-1,N)$ m.m.s. $(Y,\sfd_{Y},\mm_{Y})$ such that,  up to subsequences,  
 $$(X_{i},\sfd_{i},\mm_{i})\to (Y,\sfd_{Y},\mm_{Y}) \text{ in pmGH-sense and $\lambda_{1}(Y)=\lambda_{2}(Y)=\ldots \lambda_{N+1}(Y)=N$}. $$
 Applying   Obata's  rigidity result \cite[Theorem 1.4]{ketterer3}, we get that  $(Y,\sfd_{Y},\mm_{Y})$ is isomorphic as m.m.s. to the standard round sphere   ${\mathbb S}^{N}$ of unit radius (endowed with the standard Riemannian metric and volume measure). By Theorem \ref {thm:XiMmanIntro}  we then infer that, for large $i$ in the converging subsequence, $(X_{i},\sfd_{i}\mm_{i})$ is a non-collapsed $\RCD(N-1,N)$ space.
Moreover, by construction $\sfd_{GH} (X_{i},{\mathbb S}^{N}) \to 0$ as $i\to \infty$. Thus applying the Sphere Theorem \ref{thm:sphere}, we obtain that $(X_{i},\sfd_{i})$ is homeomorphic to ${\mathbb S}^{N}$ and $\cH^{N}(X_{i})\to  \cH^{N}(\mathbb S^{N})$. Since by Bishop-Gromov it holds also $\cH^{N}(X_{i})\leq \cH^{N}(\mathbb S^{N})$, we see that all the conclusions $(1)-(4)$ are satisfied for $\varepsilon=\varepsilon_{0}/2$ for infinitely many $i$. Contradiction.
\end{proof}

}
\appendix

\section{Almost convexity of large sets}

It is a classical result in measure theory that if $E\subset \R^{N}$ is closed and $\cH^{N-1}(E)=0$, then $\R^{N}\setminus E$ is connected. It turns out that this fact admits a natural generalization to { essentially non-branching $\MCP(K,N)$ spaces and in particular to}  $\RCD(K,N)$ spaces. The proof of this very statement was  given by Cheeger-Colding in the framework of Ricci limits \cite[Theorem 3.9]{CC00a}.

Here we prove it for general {essentially non-branching $\MCP(K,N)$ spaces}. Let us note here that our proof is inspired by but is somewhat different from the one of Cheeger and Colding.
Since this is the only result in the current paper that applies to a more general class of spaces than non-collapsed $\RCD(K,N)$ spaces,
we have placed it in an appendix.
{
\subsection{Essentially non-branching $\MCP(K,N)$ spaces: definition and basic properties}
Roughly, $\MCP(K,N)$ space are those m.m.s. $(X,\sfd,\mm)$ where Bishop-Gromov volume comparison Theorem holds (where the model space is the $N$ dimensional space form with constant Ricci curvature $K$). Here we briefly recall the important definitions, in order to make this appendix as self-contained as possible.  
\medskip

 For any $t\in [0,1]$,  let ${\rm e}_{t}$ denote the evaluation map: 
$$
  {\rm e}_{t} : \Geo(X) \to X, \qquad {\rm e}_{t}(\gamma) : = \gamma_{t}.
$$
Any geodesic $(\mu_t)_{t \in [0,1]}$ in $(\mathcal{P}_2(X), W_2)$  can be lifted to a measure $\Pi \in {\mathcal {P}}(\Geo(X))$, 
so that $({\rm e}_t)_\sharp \, \Pi = \mu_t$ for all $t \in [0,1]$. 
\\Given $\mu_{0},\mu_{1} \in \mathcal{P}_{2}(X)$, we denote by 
$\Opt(\mu_{0},\mu_{1})$ the space of all $\Pi \in \mathcal{P}(\Geo(X))$ for which $({\rm e}_0,{\rm e}_1)_\sharp\, \Pi$ 
realizes the minimum in \eqref{eq:W2def}. Such a $\Pi$ is called \emph{dynamical optimal plan}. If $(X,\sfd)$ is geodesic, then the set  $\Opt(\mu_{0},\mu_{1})$ is non-empty for any $\mu_0,\mu_1\in \mathcal{P}_2(X)$.
\medskip

A set $G \subset \Geo(X)$ is a set of non-branching geodesics if and only if for any $\gamma^{1},\gamma^{2} \in G$, it holds:
$$
\exists \;  \bar t\in (0,1) \text{ such that } \ \forall t \in [0, \bar t\,] \quad  \gamma_{ t}^{1} = \gamma_{t}^{2}   
\quad 
\Longrightarrow 
\quad 
\gamma^{1}_{s} = \gamma^{2}_{s}, \quad \forall s \in [0,1].
$$
In the appendix we will only consider essentially non-branching spaces, let us recall their definition (introduced in \cite{rajalasturm}). 
\begin{definition}\label{def:ENB}
A metric measure space $(X,\sfd, \mm)$ is \emph{essentially non-branching} (e.n.b. for short) if and only if for any $\mu_{0},\mu_{1} \in \mathcal{P}_{2}(X)$,
with $\mu_{0},\mu_{1}$ absolutely continuous with respect to $\mm$, any element of $\Opt(\mu_{0},\mu_{1})$ is concentrated on a set of non-branching geodesics.
\end{definition}

The definition of $\MCP(K,N)$ given independently by Ohta \cite{ohtmea} and Sturm \cite{sturm:II}. 
On general metric measure spaces the two definitions slightly differ, but on essentially non-branching spaces they coincide \cite[Appendix A]{Mond-Cav-17}. We use the one given in
\cite{ohtmea}.

\begin{definition}[$\MCP(K,N)$ condition]\label{def:MCP}
Let $K \in \R$ and $N \in [1,\infty)$. A metric measure space  $(X,\sfd,\mm)$ verifies $\MCP(K,N)$ if for any $\mu_{0} \in \cP_{2}(X)$ of the form 
$\mu_{0} = \frac{1}{\mm(a)} \mm\llcorner_{A}$ for some Borel set $A \subset X$ 
with $\mm(A) \in (0,\infty)$, and any $o \in X$ there exists $\Pi \in \Opt(\mu_{0},\delta_{o})$ such that 
\begin{equation}\label{eq:defMCP}
\frac{1}{\mm(A)} \mm 
 \geq  (\ee_{t})_{\sharp} \left( \tau_{K,N}^{(1-t)}(\sfd(\gamma_{0},\gamma_{1})) \Pi(d\gamma) \right), \qquad \forall \ t\in [0,1],
\end{equation}
where the  distortion coefficient $\tau_{K,N}$ was defined in \eqref{eq:deftau}.
\end{definition}

\begin{remark}\label{rem:BG}
A  key property we will use of  $\MCP(K,N)$ spaces is the validity of the Bishop-Gromov Theorem  \ref{thm:BishopGromov}, see \cite[Remark 5.3]{sturm:II} or \cite[Theorem 5.1]{ohtmea}.
\end{remark}

\begin{remark}[Notable examples of spaces fitting in the framework of the appendix]
The class of essentially non-branching $\MCP(K,N)$ spaces include many remarkable families of spaces, among them:
\begin{itemize}
\item  Smooth  Finsler manifolds where the norm on the tangent spaces is strongly convex, and which satisfy lower Ricci curvature bounds. More precisely we consider a $C^{\infty}$-manifold  $M$, endowed with a function $F:TM\to[0,\infty]$ such that $F|_{TM\setminus \{0\}}$ is $C^{\infty}$ and  for each $p \in M$ it holds that $F_p:=T_pM\to [0,\infty]$ is a  strongly-convex norm, i.e.
$$g^p_{ij}(v):=\frac{\partial^2 (F_p^2)}{\partial v^i \partial v^j}(v) \quad \text{is a positive definite matrix at every } v \in T_pM\setminus\{0\}. $$
Under these conditions, it is known that one can write the  geodesic equations and geodesics do not branch; in other words these spaces are non-branching. 
We also assume $(M,F)$ to be geodesically complete and endowed with a $C^{\infty}$  measure $\mm$ in a such a way that the associated m.m.s. $(X,F,\mm)$ satisfies the  $\MCP(K,N)$ condition,  see \cite{ohtafinsler1}. 
\item Sub-Riemannian manifolds. The following are all examples of essentially non-branching $\MCP(K,N)$-spaces:  the $(2n+1)$-dimensional Heisenberg group \cite{Juillet},  any co-rank one Carnot group \cite{Rizzi},   any ideal Carnot group \cite{Rifford}, any generalized H-type Carnot group of rank $k$ and dimension $n$ \cite{BarilariRizzi}.
\item Strong $\CD^{*}(K,N)$ spaces, and in particular $\RCD^{*}(K,N)$ (thus  also $\RCD(K,N)$) spaces  \cite{rajalasturm}. 
\end{itemize}
\end{remark}

\subsection{Disintegration in essentially non-branching $\MCP(K,N)$ spaces}
In the proof of the main result of this appendix, namely Proposition \ref{alm-all-geod}, we will use a disintegration/localization argument. In order to make the appendix as self-contained as possible, we briefly recall the results we will use.

Let $(X,\sfd,\mm)$ be an essentially non-branching m.m.s. satisfying $\MCP(K,N)$, for some $K\in \R, N\in (1,\infty)$. 
Fix a point $\bar{x}\in X$ and let $u(\cdot):=\sfd(\bar{x}, \cdot)$ be the distance function from $\bar{x}$.
Define
\begin{equation}\label{E:Gamma}  
\Gamma_{u} : = \{ (x,y) \in X\times X : u(x) - u(y) = \sfd(x,y) \}.
\end{equation}
Its transpose is given by $\Gamma^{-1}_{u}= \{ (x,y) \in X \times X : (y,x) \in \Gamma_{u} \}$. We define the \emph{transport relation} $R_u$  as:
\begin{equation}\label{E:R}
R_{u} := \Gamma_{u} \cup \Gamma^{-1}_{u}.
\end{equation}
Using that $(X,\sfd,\mm)$ is essentially non-branching, Cavalletti \cite{cava:MongeRCD}  (cf. \cite{biacava:streconv}) proved that $R_{u}$  induces a partition of $X$ (up to a subset $\mathcal N$, with   $\mm(\mathcal N)=0$) into a disjoint family (of equivalence classes) $\{X_{\alpha}\}_{\alpha \in Q}$ each of them isometric to an interval of $\R$. Here $Q$ is any set of indices.

Once an essential partition of $X$ is at disposal, a decomposition of the reference measure $\mm$ can be obtained using the Disintegration Theorem. Denote by $\tilde{X}=X\setminus \mathcal N$ the subset of full measure partitioned by $R_{u}$.  Let  $\QQ : \tilde{X} \to Q$ be the quotient map induced by the partition:
\begin{equation}\label{E:defineQ}
\alpha = \QQ(x) \iff x \in X_{\alpha}.
\end{equation}
Finally, the set of indices $Q$ can be identified with a suitable subset of $X$, intersecting each ray $X_{\alpha}$ exactly once (see \cite[Section 3.1]{CaMo} for the details), enjoying natural measurability properties.
In the next statement, we denote with  $\mathcal{M}_{+}(X)$ the space of non-negative Radon measures over $X$. 

\begin{theorem}[Theorem 3.4 and Theorem 3.6 \cite{CaMo}]\label{thm:DisintMCP}
Let $(X,\sfd,\mm)$ be an essentially non-branching m.m.s. satisfying $\MCP(K,N)$, for some $K\in \R, N\in (1,\infty)$. 
Fix a point $\bar{x}\in X$ and let $u(\cdot):=\sfd(\bar{x}, \cdot)$ be the distance function from $\bar{x}$.

Then  the measure $\mm$ admits the following disintegration formula: 
$$
\mm = \int_{Q} \mm_{\alpha} \, \qq(d\alpha),
$$
where $\qq$ is a Borel probability measure over $Q \subset X$ such that 
$\QQ_{\sharp}( \mm ) \ll \qq$ and the map 
$Q \ni \alpha \mapsto \mm_{\alpha} \in \mathcal{M}_{+}(X)$ satisfies the following properties:
\begin{itemize}
\item[(1)] for any $\mm$-measurable set $B$, the map $\alpha \mapsto \mm_{\alpha}(B)$ is $\qq$-measurable; \smallskip
\item[(2)] for $\qq$-a.e. $\alpha \in Q$, $\mm_{\alpha}$ is concentrated on $\QQ^{-1}(\alpha) = X_{\alpha}$ (strong consistency); \smallskip
\item[(3)] for any $\mm$-measurable set $B$ and $\qq$-measurable set $C$, the following disintegration formula holds: 
$$
\mm(B \cap \QQ^{-1}(C)) = \int_{C} \mm_{\alpha}(B) \, \qq(d\alpha);
$$
\item[(4)] for $\qq$-a.e. $\alpha$, $\mm_{\alpha}$ is a Radon measure  with $\mm_{\alpha}=h_{\alpha} \cH^{1}\llcorner_{X_{\alpha}} \ll \cH^{1}\llcorner_{X_{\alpha}}$
and $(\bar X_{\alpha},\sfd,\mm_{\alpha})$ verifies $\MCP(K,N)$. 
\end{itemize}
\end{theorem}

}

\subsection{The result}

Let $(X,\sfd, \mm)$ be a metric measure space. For any $\beta\in \R$  we can consider a codimension $\beta$ version of $\mm$ denoted by $\mm_{-\beta}$ as defined by Cheeger and Colding in \cite[Section 2]{CC00a}.
Recall that it's defined as follows:

For $\delta>0$ set 
\[
(\mm_{-\beta})_\delta(U)=\inf_{\mathcal B} \sum_i r_i^{-\beta}\mm(B_{r_i}(q_i))
\]
where $\mathcal B=\{B_{r_i}(q_i)\}$ is a collection of balls covering $U$ with all $r_i\le \delta$.
Then $(\mm_{-\beta})_\delta(U)$ is non-increasing in $\delta$ and we put

\[
\mm_{-\beta}(U)=\lim_{\delta\to0+}(\mm_{-\beta})_\delta(U)
\]
This obviously defines a metric outer measure and hence all Borel subsets of $X$ are $\mm_{-\beta}$ measurable.

\begin{proposition}\label{alm-all-geod}
Let $(X,\sfd,\mm)$ be an { essentially non-branching $\MCP(K,N)$ space}. Let $S\subset X$ be a closed subset with $\mm_{-1}(S)=0$. Let $x_1\in X\backslash S$.
Then for $\mm$-a.e. $y\in X\backslash S$  there exists a geodesic joining $x_{1}$ and $y$ which is entirely contained in $X\backslash S$.
\end{proposition}

We first establish the following preliminary lemma
which generalizes and strengthens \cite[Lemma 3.1]{CC00a} in the  $\MCP$ case. 

\begin{lemma}\label{lem-vol-est-col}  
Given $\delta,d,N>0,K\in\R$ there is $C(\delta,d,N,K)>0$ such that the following holds.

Let $(X,\sfd,\mm)$ be an  { essentially non-branching $\MCP(K,N)$ space}. Assume there exist $x_1,x_2\in X$ with $B_{2\delta}(x_1)\cap B_{2\delta}(x_2)=\emptyset$ and satisfying the following conditions:

\begin{itemize}
\item  Denoting $$E=\cup_{j=1}^lB_{r_j}(q_j)$$ the union of finitely many balls in $X$, it holds that
\[
B_{2\delta}(x_1)\cup B_{2\delta}(x_2)\subset B_d(x_1)\backslash E.
\]
\item There is a subset $Y\subset  B_\delta(x_2)$ with $\mm(Y)\ge \frac{1}{2}\mm( B_\delta(x_2))$ such that for every $x\in Y$,  every geodesic from $x_{1}$ to $x$   intersects $E$.
\end{itemize}
Then
\[
0<C(\delta,d,N,K)<\sum_j\frac{\mm( B_{r_j}(q_j))}{r_j \mm(B_d(x_1))}.
\]
\end{lemma}
\begin{proof}

\textbf{Step 1}. The key step in the proof of the lemma is the following inequality

\begin{equation}\label{content-ineq-col}
\mm( B_\delta(x_2))\le C(\delta,d,N,K) M_{+}(\partial E),
\end{equation}
where we denote with $ M_{+}(\partial E)$ the co-dimension one Minkowski content defined as
\begin{equation}\label{def:MinkCont}
M_{+}(\partial E)=\liminf_{\ve\downarrow 0} \frac{\mm(U_\ve)}{\ve},
\end{equation}
where $U_{\ve}:=\{x \in X \,:\, \exists y \in  \partial E \, \text{ such that } \, \sfd(x,y)< \ve \}$ is the $\ve$-neighborhood of $\partial E$ with respect to the metric $\sfd$.

 Applying  Theorem \ref{thm:DisintMCP}, we obtain  the radial disintegration of the background measure $\mm$ with respect to the point $x_1$:

\[
\mm=\int_Q h_\alpha \mathcal H^1\llcorner_{X_\alpha}\mathfrak q(d\alpha)
\]
where $\mm_\alpha= h_\alpha  \cH^1\llcorner_{X_\alpha} \ll \cH^1\llcorner_{X_\alpha}$ satisfies $\MCP(K,N)$.

Let $0<\eps<\delta$ and let $U_\eps$ the $\eps$-tubular neighbourhood of $\partial E$ defined above.

Let $\mathcal A'\subset Q$ be the set of all indices $\alpha$ such that $X_\alpha \cap Y\ne\emptyset$.
For any $\alpha\in\mathcal A'$ let $x_\alpha\in X_\alpha$ be the  point of intersection of $X_\alpha$ with $\partial E$ (note that $X_\alpha\cap\partial E\ne\emptyset$ by the assumptions of the lemma) which is closest to $x_1$.

For each $\alpha\in\mathcal A'$ let $J_\alpha=X_\alpha\cap Y$ and let $I_\alpha=B_\eps(x_\alpha)\cap X_\alpha$. Note that $I_\alpha\subset U_\eps, \mathcal H^1(I_\alpha)=2\eps$ and $\mathcal H^1(J_\alpha)\le 2\delta$ for any $\alpha\in\mathcal A'$.

Also note that $\sfd(x_\alpha,x_1)\ge 2\delta$. Now  the $\MCP(K,N)$ condition implies (see for instance \cite[(2.10)]{CaMo})  that for any $x\in I_\alpha, y\in J_\alpha$  the densities $h_\alpha$ at these points satisfy

\[
h_\alpha(x)\ge C(K,N,d,\delta) h_\alpha(y).
\]
Averaging this inequality over $I_\alpha, J_\alpha$ with respect to $\mathcal H^1$ gives that

\[
\mm_\alpha(I_\alpha)\ge \eps \frac{C(\delta,N,K,d) \mm_\alpha(J_\alpha)}{\delta}
\]
Integrating the last estimate with respect to $\mathfrak q(d\alpha)$ and taking into account that  $\cup_{\alpha} I_\alpha \subset U_\eps$ gives

\[
\mm(U_\eps)\ge \eps \frac{c(\delta, N,K,d) \,  \mm(Y)}{\delta}\ge \eps c(\delta, N,K,d) \, {\mm(B_\delta(x_2))}.
\]
Dividing by $\eps$ and sending $\eps\to 0$ gives \eqref{content-ineq-col}.
\\

\textbf{Step 2}. The estimate \eqref{content-ineq-col} gives

\begin{align*}
 \mm( B_\delta(x_2)) &\le C(\delta,d,N,K) M_{+}(\partial E) \\
& \le C(\delta,d,N,K) \sum_j M_{+} (\partial B_{r_j}(q_j))\le C(\delta,d,N,K) \sum_j \frac{\mm( B_{r_j}(q_j))}{r_j}
\end{align*}
where in the last estimate we used Bishop-Gromov Theorem \ref{thm:BishopGromov} (see Remark \ref{rem:BG}). 
Dividing by $\mm(B_d(x_1))$ and using Bishop-Gromov Theorem again we get 
\[
c(\delta,d,N,K)\le \frac{ \mm( B_\delta(x_2))}{\mm(B_{2d}(x_2))}\le \frac{ \mm( B_\delta(x_2))}{\mm(B_{d}(x_1))}\le  C(\delta,d,N,K) \sum_j \frac{\mm( B_{r_j}(q_j))}{r_j \mm(B_{d}(x_1))}
\]
which finishes the proof of the lemma.

\end{proof}

\begin{proof}[Proof of Proposition~\ref{alm-all-geod}]

It's enough to prove the proposition for compact $S$. 
 Let $Y$ be the set of points $y\in X\backslash S$ such that every geodesic from $x_{1}$ to $y$ intersects $S$. Suppose $\mm(Y)>0$. Let $x_2\in X\backslash S$ be a Lebesgue-density point of $Y$. {Note that $x_{2}\neq x_{1}$ since $S$ is closed and $x_1\notin S$. }Let $d>0$ be big enough so that  $\{x_2\}\cup S\subset B_d(x_1)$. 
 
  Let $\delta>0$ be small enough so  that $\mm(B_\delta(x_2)\cap Y)\ge \frac{1}{2}\mm(B_\delta(x_2))$ and $B_{10\delta}(x_1)\cup B_{10\delta}(x_2)\subset  B_d(x_1)\backslash S$ and $B_{10\delta}(x_1)\cap B_{10\delta}(x_2)=\emptyset$. Since  $\mm_{-1}(S)=0$ and $S$ is compact,  for any $\eta>0$ there exists a finite collection of balls $\{B_{r_j}(q_j)\}_{j=1}^l$ such that all $r_j\le\delta$ and 
\[
 \sum_j\frac{\mm( B_{r_j}(q_j))}{r_j}\le \eta.
 \] 
 On the other hand, applying  Lemma~\ref{lem-vol-est-col} to $x_1,x_2$,  and $E_\eta=\cup_jB_{r_j}(q_j)$ we get that 
 
 \[
0<C(\delta,d,N,K) \, \mm(B_d(x_1))<\sum_j\frac{\mm(B_{r_j}(q_j))}{r_j}.
\]
This gives a contradiction when $\eta$ is sufficiently small.
\end{proof}

\begin{corollary}\label{non-collapsed-alm-all-geod}
Let $(X,\sfd,\cH^{N})$ be a non-collapsed $\RCD(K,N)$ space. Let $B\subset X$ be a closed subset with $\cH^{N-1}(B)=0$. Let $x_1\in X\backslash B$.
Then for $\mathcal H^{N}$-a.e. $y\in X\backslash B$  there exists a geodesic from $x_{1}$ to $y$  which is contained in $X\backslash B$.
\end{corollary}
\begin{proof}
Since the measure $\mm=\cH^{N}$ is locally Ahlfors regular on a non-collapsed $\RCD(K,N)$ space, it easily follows from the definition of $\mm_{-1}$ that if $\cH^{N-1}(B)=0$ then $\mm_{-1}(B)=0$ as well.
Now the result follows from Proposition~\ref{alm-all-geod}\\
\end{proof}

\small{
\bibliographystyle{amsalpha}
\bibliography{RCDHN-bib}

\providecommand{\bysame}{\leavevmode\hbox to3em{\hrulefill}\thinspace}
\providecommand{\MR}{\relax\ifhmode\unskip\space\fi MR }
\providecommand{\MRhref}[2]{%
  \href{http://www.ams.org/mathscinet-getitem?mr=#1}{#2}
}
\providecommand{\href}[2]{#2}
\begin{thebibliography}{GGKMS18}

\bibitem[ABS19]{Ant-Brue-Sem-19}
Gioacchino Antonelli, Elia Bru{\`e}, and Daniele Semola, \emph{Volume bounds
  for the quantitative singular strata of non collapsed {RCD} metric measure
  spaces}, Anal. Geom. Metr. Spaces \textbf{7} (2019), no.~1, 158--178.

\bibitem[AGMR15]{AGMR12}
Luigi Ambrosio, Nicola Gigli, Andrea Mondino, and Tapio Rajala,
  \emph{Riemannian {R}icci curvature lower bounds in metric measure spaces with
  {$\sigma$}-finite measure}, Trans. Amer. Math. Soc. \textbf{367} (2015),
  no.~7, 4661--4701. \MR{3335397}

\bibitem[AGS14]{AGS}
Luigi Ambrosio, Nicola Gigli, and Giuseppe Savar{\'e}, \emph{Metric measure
  spaces with {R}iemannian {R}icci curvature bounded from below}, Duke Math. J.
  \textbf{163} (2014), no.~7, 1405--1490. \MR{3205729}

\bibitem[AHPT18]{AHPT}
Luigi Ambrosio, Shouhei Honda, Jacobus Portegies, and David Tewodrose,
  \emph{Embedding of ${RCD}^{*}({K,N})$ spaces in ${L}^2$ via eigenfunctions},
  arXiv:1812.03712v1 (2018).

\bibitem[Amb18]{AmbrosioICM}
Luigi Ambrosio, \emph{Calculus and curvature-dimension bounds in metric measure
  spaces}, Proc. Int. Cong. of Math. \textbf{1} (2018), 301--340.

\bibitem[AMS19]{AMS}
Luigi Ambrosio, Andrea Mondino, and Giuseppe Savar{\'e}, \emph{Nonlinear
  diffusion equations and curvature conditions in metric measure spaces}, Mem.
  Amer. Math. Soc. \textbf{262} (2019), no.~1270, v+121.

\bibitem[BC13]{biacava:streconv}
Stefano Bianchini and Fabio Cavalletti, \emph{The {M}onge problem for distance
  cost in geodesic spaces}, Commun. Math. Phys. \textbf{318} (2013), 615 --
  673.

\bibitem[Ber07]{BertrandEigenvalue}
J\'{e}r\^{o}me Bertrand, \emph{Pincement spectral en courbure de ricci
  positive}, Comment. Math. Helv \textbf{82} (2007), no.~2, 323--352.

\bibitem[BR18]{BarilariRizzi}
Davide Barilari and Luca Rizzi, \emph{Sharp measure contraction property for
  generalized $h$-type carnot groups}, Commun. Contemp. Math. \textbf{20,}
  (2018), no.~6, 1750081.

\bibitem[BS10]{BS10}
Kathrin Bacher and Karl-Theodor Sturm, \emph{Localization and tensorization
  properties of the curvature-dimension condition for metric measure spaces},
  J. Funct. Anal. \textbf{259} (2010), no.~1, 28--56. \MR{2610378
  (2011i:53050)}

\bibitem[BS19]{BS18}
Elia Bru{\`e} and Daniele Semola, \emph{{Constancy of the dimension for
  RCD(K,N) spaces via regularity of Lagrangian flows}}, Comm. Pure and Applied
  Math. (2019), DOI: 10.1002/cpa.21849.

\bibitem[Cav14]{cava:MongeRCD}
Fabio Cavalletti, \emph{Monge problem in metric measure spaces with riemannian
  curvature-dimension condition}, Nonlinear Analysis TMA \textbf{99} (2014),
  136--151.

\bibitem[CC97]{CC97}
Jeff Cheeger and Tobias~H. Colding, \emph{On the structure of spaces with
  {R}icci curvature bounded below. {I}}, J. Differential Geom. \textbf{46}
  (1997), no.~3, 406--480. \MR{1484888 (98k:53044)}

\bibitem[CC00a]{CC00a}
\bysame, \emph{On the structure of spaces with {R}icci curvature bounded below.
  {II}}, J. Differential Geom. \textbf{54} (2000), no.~1, 13--35. \MR{1815410
  (2003a:53043)}

\bibitem[CC00b]{CC00b}
\bysame, \emph{On the structure of spaces with {R}icci curvature bounded below.
  {III}}, J. Differential Geom. \textbf{54} (2000), no.~1, 37--74. \MR{1815411
  (2003a:53044)}

\bibitem[CJN18]{cheeger-jiang-naber-18}
Jeff Cheeger, Wenshuai Jiang, and Aaron Naber, \emph{Rectifiability of singular
  sets in noncollapsed spaces with ricci curvature bounded below},
  arXiv:1805.07988, 2018.

\bibitem[CM]{CaMo}
Fabio Cavalletti and Andrea Mondino, \emph{New formulas for the laplacian of
  distance functions and applications}, arXiv:1803.09687, to appear in
  {Analysis $\&$ PDE}.

\bibitem[CM16a]{CaMi}
Fabio Cavalletti and Emanuel Milman, \emph{{The Globalization Theorem for the
  Curvature Dimension Condition}}, arXiv:1612.07623, 2016.

\bibitem[CM16b]{CaMoAIM}
Fabio Cavalletti and Andrea Mondino, \emph{Measure rigidity of ricci curvature
  lower bounds}, Advances in Math. \textbf{286} (2016), 430--480.

\bibitem[CM17]{Mond-Cav-17}
Fabio Cavalletti and Andrea Mondino, \emph{Optimal maps in essentially
  non-branching spaces}, Commun. Contemp. Math. \textbf{19} (2017), no.~6,
  1750007, 27. \MR{3691502}

\bibitem[CN13]{CheNab13}
Jeff Cheeger and Aaron Naber, \emph{Lower bounds on {R}icci curvature and
  quantitative behavior of singular sets}, Invent. Math. \textbf{191} (2013),
  no.~2, 321--339. \MR{3010378}

\bibitem[CN15]{Ch-Na-codim4}
\bysame, \emph{Regularity of {Einstein} manifolds and the codimension 4
  conjecture}, Annals of Mathematics (2015), 1093--1165 (en).

\bibitem[Col96a]{Col96b}
Tobias~H. Colding, \emph{Large manifolds with positive {R}icci curvature},
  Invent. Math. \textbf{124} (1996), no.~1-3, 193--214. \MR{1369415}

\bibitem[Col96b]{Col96a}
\bysame, \emph{Shape of manifolds with positive {R}icci curvature}, Invent.
  Math. \textbf{124} (1996), no.~1-3, 175--191. \MR{1369414 (96k:53067)}

\bibitem[Col96c]{coldingshape}
\bysame, \emph{Shape of manifolds with positive {R}icci curvature}, Invent.
  Math. \textbf{124} (1996), no.~1-3, 175--191. \MR{1369414 (96k:53067)}

\bibitem[Col97]{Col97}
\bysame, \emph{Ricci curvature and volume convergence}, Ann. of Math. (2)
  \textbf{145} (1997), no.~3, 477--501. \MR{1454700}

\bibitem[DPG18]{GDPNonColl}
Guido De~Philippis and Nicola Gigli, \emph{Non-collapsed spaces with {R}icci
  curvature bounded from below}, J. \'{E}c. polytech. Math. \textbf{5} (2018),
  613--650. \MR{3852263}

\bibitem[DPMR17]{DePRinMar}
Guido De~Philippis, Andrea Marchese, and Filip Rindler, \emph{On a conjecture
  of {C}heeger}, Measure theory in non-smooth spaces, Partial Differ. Equ.
  Meas. Theory, De Gruyter Open, Warsaw, 2017, pp.~145--155. \MR{3701738}

\bibitem[DS17]{Donaldson-Sun}
Simon Donaldson and Song Sun, \emph{Gromov-{H}ausdorff limits of {K}\"{a}hler
  manifolds and algebraic geometry, {II}}, J. Differential Geom. \textbf{107}
  (2017), no.~2, 327--371. \MR{3707646}

\bibitem[EKS15]{EKS}
Matthias Erbar, Kazumasa Kuwada, and Karl-Theodor Sturm, \emph{On the
  equivalence of the entropic curvature-dimension condition and {B}ochner's
  inequality on metric measure spaces}, Invent. Math. \textbf{201} (2015),
  no.~3, 993--1071. \MR{3385639}

\bibitem[ES]{ErSt18}
M~Erbar and K.T. Sturm, \emph{Rigidity of cones with bounded ricci curvature},
  arxiv:1712.08093, to appear in Journ. Europ. Math. Soc.

\bibitem[GGKMS18]{GGKMS}
Fernando Galaz-Garcia, Martin Kell, Andrea Mondino, and Gerardo Sosa, \emph{On
  quotients of spaces with ricci curvature bounded below}, J. Funct. Anal.
  \textbf{275} (2018), no.~6, 1368--1446.

\bibitem[Gig14]{giglisplittingshort}
Nicola Gigli, \emph{An overview on the proof of the splitting theorem in a
  non-smooth context}, Analysis and Geometry in Metric Spaces \textbf{2}
  (2014), no.~1, 169--213.

\bibitem[Gig15]{gigli:laplacian}
\bysame, \emph{On the differential structure of metric measure spaces and
  applications}, Mem. Amer. Math. Soc. \textbf{236} (2015), no.~1113, vi+91.
  \MR{3381131}

\bibitem[GMS15]{GMS2013}
Nicola Gigli, Andrea Mondino, and Giuseppe Savar{\'e}, \emph{Convergence of
  pointed non-compact metric measure spaces and stability of {R}icci curvature
  bounds and heat flows}, Proc. Lond. Math. Soc. (3) \textbf{111} (2015),
  no.~5, 1071--1129. \MR{3477230}

\bibitem[GP16]{GiPa}
Nicola {Gigli} and Enrico {Pasqualetto}, \emph{{Behaviour of the reference
  measure on $RCD$ spaces under charts}}, arXiv e-prints (2016),
  arXiv:1607.05188.

\bibitem[Han19]{HanMeasRig}
Bang-Xian Han, \emph{Measure rigidity of synthetic lower ricci curvature bound
  on riemannian manifolds}, Preprint arXiv:1902.00942 (2019).

\bibitem[HM]{HM}
Shouhei Honda and Ilaria Mondello, \emph{Sphere theorems for {RCD} and
  stratified spaces}, arXiv:1907.03482, to appear in Ann. Sc. Norm. Super. Pisa
  Cl. Sci.

\bibitem[Hon]{HondaNC}
Shouhei Honda, \emph{New differential operator and non-collapsed rcd spaces},
  arXiv:1905.00123v2, to appear in {Geom. $\&$ Topol.}

\bibitem[Jui09]{Juillet}
Nicolas Juillet, \emph{Geometric inequalities and generalized ricci bounds in
  the heisenberg group}, Int. Math. Res. Not. \textbf{13} (2009), 2347--2373.

\bibitem[Kap07]{Kap07}
Vitali Kapovitch, \emph{Perelman's stability theorem}, Surveys in Differential
  Geometry, vol.~XI, International Press, Boston, MA, 2007, A supplement to the
  Journal of Differential Geometry, pp.~103--136.

\bibitem[Ket15a]{Ket}
Christian Ketterer, \emph{Cones over metric measure spaces and the maximal
  diameter theorem}, J. Math. Pures Appl. (9) \textbf{103} (2015), no.~5,
  1228--1275. \MR{3333056}

\bibitem[Ket15b]{ketterer3}
\bysame, \emph{Obata's rigidity theorem for metric measure spaces}, Anal. Geom.
  Metr. Spaces \textbf{3} (2015), 278--295. \MR{3403434}

\bibitem[Kit17]{Kitab-noncol}
Yu~Kitabeppu, \emph{A {B}ishop-type inequality on metric measure spaces with
  {R}icci curvature bounded below}, Proc. Amer. Math. Soc. \textbf{145} (2017),
  no.~7, 3137--3151. \MR{3637960}

\bibitem[Kit19]{KitaPOTA}
\bysame, \emph{A sufficient condition to a regular set being of positive
  measure on {RCD} spaces}, Potential Analysis \textbf{51} (2019), 179--196.

\bibitem[KK]{Kap-Ket-18}
Vitali Kapovitch and Christian Ketterer, \emph{{CD} meets {CAT}},
  arXiv:1712.02839, to appear in Crelle's Journal,
  DOI:10.1515/crelle-2019-0021.

\bibitem[KL16]{KL}
Yu~Kitabeppu and Sajjad Lakzian, \emph{Characterization of low dimensional
  {$RCD^*(K,N)$} spaces}, Anal. Geom. Metr. Spaces \textbf{4} (2016), 187--215.
  \MR{3550295}

\bibitem[KLP]{KaLytPe}
Vitali Kapovitch, Alexander Lytchak, and Anton Petrunin, \emph{Metric-measure
  boundary and geodesic flow on {A}lexandrov spaces}, arXiv:1705.04767, to
  appear in { Journ. Europ. Math. Soc.}

\bibitem[KM18]{KM}
Martin Kell and Andrea Mondino, \emph{On the volume measure of non-smooth
  spaces with {R}icci curvature bounded below}, Ann. Sc. Norm. Super. Pisa Cl.
  Sci. (5) \textbf{18} (2018), no.~2, 593--610. \MR{3801291}

\bibitem[KS77]{KS}
Robion Kirby and Laurence Siebenmann, \emph{Foundational essays on topological
  manifolds, smoothings, and triangulations}, Princeton University Press,
  Princeton, N.J., 1977, With notes by John Milnor and Michael Atiyah, Annals
  of Mathematics Studies, No. 88.

\bibitem[LS18]{LS}
Alexander Lytchak and Stephan Stadler, \emph{Ricci curvature in dimension 2},
  arXiv:1812.08225, 2018.

\bibitem[LV09]{lottvillani:metric}
John Lott and C{\'e}dric Villani, \emph{Ricci curvature for metric-measure
  spaces via optimal transport}, Ann. of Math. (2) \textbf{169} (2009), no.~3,
  903--991. \MR{2480619}

\bibitem[MN19]{MN}
Andrea Mondino and Aaron Naber, \emph{{Structure Theory of Metric-Measure
  Spaces with Lower Ricci Curvature Bounds}}, Journal of European Math. Soc.
  \textbf{21} (2019), no.~6, 1809--1854.

\bibitem[Oht07]{ohtmea}
Shin-ichi Ohta, \emph{On the measure contraction property of metric measure
  spaces}, Comment. Math. Helv. \textbf{82} (2007), no.~4, 805--828.
  \MR{2341840 (2008j:53075)}

\bibitem[Oht09]{ohtafinsler1}
\bysame, \emph{Finsler interpolation inequalities}, Calc. Var. Partial
  Differential Equations \textbf{36} (2009), no.~2, 211--249. \MR{2546027
  (2011m:58027)}

\bibitem[Per91]{Per-stab}
G.~Perelman, \emph{A.{D}. {A}lexandrov's spaces with curvatures bounded from
  below, {II}},
  http://www.math.psu.edu/petrunin/papers/alexandrov/perelmanASWCBFB2+.pdf,
  1991.

\bibitem[Pet98]{petruninsecondvariation}
Anton Petrunin, \emph{Parallel transportation for {A}lexandrov space with
  curvature bounded below}, Geom. Funct. Anal. \textbf{8} (1998), no.~1,
  123--148. \MR{1601854 (98j:53048)}

\bibitem[Pet99]{PetEigenvalue}
Peter Petersen, \emph{On eigenvalue pinching in positive curvature}, Invent.
  Math. \textbf{138} (1999), no.~1--21.

\bibitem[Pet11]{palvs}
Anton Petrunin, \emph{Alexandrov meets {L}ott-{V}illani-{S}turm}, M\"unster J.
  Math. \textbf{4} (2011), 53--64. \MR{2869253 (2012m:53087)}

\bibitem[Qui82]{QuinnIII}
Frank Quinn, \emph{Ends of maps. {III}. {D}imensions {$4$} and {$5$}}, J.
  Differential Geom. \textbf{17} (1982), no.~3, 503--521.

\bibitem[Rif13]{Rifford}
Ludovic Rifford, \emph{Ricci curvatures in carnot groups}, Math. Control Relat.
  Fields \textbf{3} (2013), no.~4, 467--487.

\bibitem[Riz16]{Rizzi}
Luca Rizzi, \emph{Measure contraction properties of carnot groups}, Calc. Var.
  Partial Differential Equations \textbf{55} (2016), no.~3.

\bibitem[RS14]{rajalasturm}
Tapio Rajala and Karl-Theodor Sturm, \emph{Non-branching geodesics and optimal
  maps in strong {$CD(K,\infty)$}-spaces}, Calc. Var. Partial Differential
  Equations \textbf{50} (2014), no.~3-4, 831--846. \MR{3216835}

\bibitem[Sie72]{Sieb}
L.~C. Siebenmann, \emph{Deformation of homeomorphisms on stratified sets. {I},
  {II}}, Comment. Math. Helv. \textbf{47} (1972), 123--136; ibid. 47 (1972),
  137--163.

\bibitem[Sim12]{Sim}
Miles Simon, \emph{Ricci flow of non-collapsed three manifolds whose ricci
  curvature is bounded from below}, J. Reine Angew. Math. \textbf{662} (2012),
  59--94.

\bibitem[ST17]{SimTop}
Miles Simon and Peter~M. Topping, \emph{Local mollification of {R}iemannian
  metrics using {R}icci flow, and {R}icci limit spaces}, Preprint
  arXiv:1706.09490 (2017).

\bibitem[Stu06a]{sturm:I}
Karl-Theodor Sturm, \emph{On the geometry of metric measure spaces. {I}}, Acta
  Math. \textbf{196} (2006), no.~1, 65--131. \MR{2237206 (2007k:53051a)}

\bibitem[Stu06b]{sturm:II}
\bysame, \emph{On the geometry of metric measure spaces. {II}}, Acta Math.
  \textbf{196} (2006), no.~1, 133--177. \MR{2237207 (2007k:53051b)}

\bibitem[Vil09]{Vil}
C{\'e}dric Villani, \emph{Optimal transport, old and new}, Grundlehren der
  Mathematischen Wissenschaften [Fundamental Principles of Mathematical
  Sciences], vol. 338, Springer-Verlag, Berlin, 2009. \MR{2459454
  (2010f:49001)}

\bibitem[Vil17]{VilB}
Cedric Villani, \emph{In\'egalit\'es isop\'erim\'etriques dans les espaces
  m\'etriques mesur\'es [d'apr\`es f. cavalletti \& a. mondino]}, S\'eminaire
  BOURBAKI 69\`me, Available at http://www.bourbaki.ens.fr/TEXTES/1127.pdf.
  (2016--2017), no.~1127.

\bibitem[Won08]{Wong}
Jeremy Wong, \emph{An extension procedure for manifolds with boundary}, Pac. J.
  Math \textbf{235} (2008), no.~1, 173--199.

\end{thebibliography}

}
\end{document}